\documentclass[12pt]{article}
\usepackage{color}

\usepackage{amsmath}
\usepackage{amsfonts}
\usepackage{amsthm}
\usepackage{amssymb}
\usepackage{bm}
\usepackage{enumerate}
\usepackage{marvosym}
\usepackage{comment}

\makeatletter

\newtheoremstyle{mystyle}{}{}{\slshape}{2pt}{\scshape}{.}{ }{}

\newtheorem{thm}{Theorem}[section]
\newtheorem{cor}[thm]{Corollary}

\newtheorem{prop}[thm]{Proposition}
\newtheorem{lemme}[thm]{Lemma}
\newtheorem{lem}[thm]{Lemma}

\newtheorem{fait}[thm]{Fact}

\newtheorem{question}[thm]{Question}

\newtheorem{fact}[thm]{Fact}
\theoremstyle{definition}
\newtheorem{defi}[thm]{Definition}
\theoremstyle{mystyle}

\theoremstyle{remark}
\newtheorem{rem}[thm]{Remark}

\newtheorem*{claim}{Claim}

\newtheorem*{claim1}{Claim 1}
\newtheorem*{claim1'}{Claim 1'}
\newtheorem*{claim2}{Claim 2}
\newtheorem*{claim3}{Claim 3}
\newtheorem*{claim4}{Claim 4}
\newtheorem*{claim5}{Claim 5}

\newenvironment{claimproof}
    {\begin{proof}}{ \end{proof}}

\newcommand{\monster}{\mathcal U}
\newcommand{\ntp}{NTP$_2$ }

\newcommand{\PRCB}{T}

\newcommand{\LC}{\mathcal{L}}
\newcommand{\LCR}{\mathcal{L_{\text{ring}}}}

\newcommand{\Li}{\mathcal{L}^{(i)}_{\text{ring}}}

\newcommand{\clos}[2]{#1^{(#2)}}

\newcommand{\acl}{\mathrm{acl}}
\newcommand{\dcl}{\mathrm{dcl}}
\newcommand{\tp}{\mathrm{tp}}
\newcommand{\qftp}{\mathrm{qftp}}

\newcommand{\cG}{\cdot_G}
\newcommand{\ignore}[1]{}

\def\OK{\overline{K}}
\newcommand{\nf}{\times_{nf}}

\def\indsym#1#2{%
 \setbox0=\hbox{$\m@th#1x$}%
 \kern\wd0%
 \hbox to 0pt{\hss$\m@th#1\mid$\hbox to 0pt{$\m@th#1^{#2}$\hss}\hss}%
 \lower.9\ht0\hbox to 0pt{\hss$\m@th#1\smile$\hss}%
 \kern\wd0}
\newcommand{\ind}[1][]{\mathop{\mathpalette\indsym{#1}}}

\def\nindsym#1#2{%
 \setbox0=\hbox{$\m@th#1x$}%
 \kern\wd0%
 \hbox to 0pt{\hss$\m@th#1\not$\kern1.4\wd0\hss}
 \hbox to 0pt{\hss$\m@th#1\mid$\hbox to 0pt{$\m@th#1^{#2}$\hss}\hss}%
 \lower.9\ht0\hbox to 0pt{\hss$\m@th#1\smile$\hss}%
 \kern\wd0}

\title{Stabilizers, groups with f-generics in NTP$_2$ and PRC fields}
\author{Samaria Montenegro\footnotemark[1]  \and Alf Onshuus\thanks{Partially supported by Colciencias grant number 120471250707}
 \and Pierre Simon\footnote{Partially supported by ValCoMo
(ANR-13-BS01-0006), NSF (grant DMS
1665491), and the Sloan foundation.}}
\date{}

\begin{document}

\maketitle

\begin{abstract}In this paper we develop three different subjects. We study and
prove alternative versions of Hrushovski's ``Stabilizer Theorem'',
we generalize part of the basic theory of definably amenable NIP
groups to NTP$_2$ theories, and finally, we use all this machinery
to study groups with f-generic types definable in bounded PRC
fields.



\end{abstract}

\section{Introduction}

This paper has three main parts, each of which we believe may be
of independent interest.

Section 2 is very much self contained, and only requires knowledge
of basic concepts of model theory, all of which are contained in
(for instance) the introduction of \cite{HrPi}. It is devoted to
the study of S1 ideals and various versions of Hrushovski's
Stabilizer Theorem (Fact \ref{fact_stab}). We prove two variations
on it. Theorem \ref{th_babystab} is very close to Hrushovski's
original theorem but
 allows what we feel is a simpler proof and more natural hypothesis. Theorem
\ref{th_stabilizer} on the other hand, substantially weakens the
hypothesis on the S1 ideal. We use it to generalize several
results in \cite{HrPi} and \cite{Ba} by proving an ``algebraic
group chunk theorem'' (Theorem \ref{algebraic group chunk}) in
many geometric theories. Reading \cite{Ba} and conversations with
Barriga made us realize that Theorem \ref{algebraic group chunk}
implied that every torsion free group definable in a real closed
field $R$ is semi-algebraically isomorphic to the $R$ points of an
algebraic group $H$, a result we believe was previously unknown
which we include as a corollary.

In Section \ref{SGroups} we prove some results about groups
definable in an NTP$_2$ theory admitting f-generic types. We
generalize some basic statements proved in \cite{CS} for definably
amenable NIP groups. Apart from the use of Theorem \ref{th_babystab}, this section
is self contained.

The rest of the paper is devoted to the study of groups definable in
bounded PRC fields.

A field is PRC if every absolutely irreducible variety which has
zeros in every real closed extension has a zero in the field.
Hence PRC fields generalize both the notions of real closed fields
and of pseudo algebraically closed fields (PAC). It is shown in
\cite{Mon} that bounded PRC fields are NTP$_2$, a notion which
generalizes the better known concepts of dependent theories and
simple theories.   Since bounded PAC fields have been a very
inspirational example of a simple unstable field,  and real closed
fields are one of the main examples of dependent fields, bounded
PRC fields are examples of NTP$_2$ fields, the study of which
might enable us to predict which properties can and cannot hold in
an NTP$_2$ theory.

In Section \ref{STheorem} we try to understand definable groups in a bounded
PRC field, assuming in addition existence of f-generic types. We prove
that such a group is isogeneous to a finite index subgroup of a
quantifier-free definable group (Theorem \ref{th_main}). In fact,
the latter group admits a definable covering by multi-cells on
which the group operation is algebraic. This generalizes similar
results proved in \cite{HrPi} by Hrushovski and Pillay for (not
necessarily f-generic) groups definable in both pseudofinite
fields and real closed fields. Our theorem applies in particular
to all solvable groups.

In Section \ref{SPRC} we recall some results on PRC fields and
prove that the expansion of a bounded PRC field obtained by adding
all quantifier-free externally definable sets has elimination of
quantifiers.

The sketch of the proof of the main theorem is as follows: After
an initial reduction to groups of finite index, we use Theorem
\ref{algebraic group chunk} to show that given a group $G$ with
f-generics definable in a bounded PRC field, there is an
algebraic group $H$ and a (relatively) definable isomorphism
between type-definable subgroups $G^{00}_M$ of $G$ and $K$ of $H$
where $G^{00}_M$ is the maximum type definable over $M$ subgroup
of $G$. In Section \ref{SAlgebraic} we show that for any such $K$
(a type definable subgroup of an algebraic group), if
$\overline{K}$ denotes the topological closure of  $K$, then
$\overline{K}/K$ is profinite. The proof then continues adapting
the proofs in \cite{HrPi} for the pseudo-finite case and for the
real closed case.

\section{S1 ideals and Stabilizer theorems}\label{SStabilizer}

Let $M$ be a model and let $G$ be an $M$-definable group. Let
$\mu$ be an $M$-invariant ideal of definable subsets of $G$ which
is invariant by left translations by elements of $G$. We say that
a type $p(x)$ in $G$ is \emph{$\mu$-wide} if it is not contained
in a set $D\in \mu$. If the ideal $\mu$ is fixed and no confusion
can arise, we will refer to $\mu$-wide types as ``wide''.

A key concept we will need is Hrushovski's definition of an S1
ideal.

\begin{defi}
An $A$-invariant ideal $\mu$ has the \emph{S1 property} if whenever $(
a_j)_{j \in \omega}$ is an $A$-indiscernible sequence and
$\phi(x,y)$ is a formula, then if $\phi(x,a_i)\wedge \phi(x,a_j)$
is in $\mu$ for some/all $i\neq j$, then $\phi(x,a_i)$ is in $\mu$
for some/all $i$.

We will say that the ideal $\mu$ is \emph{S1 on the $A$-definable set
$X$} if $X$ is not in $\mu$ and the property above holds for
formulas $\phi(x,a_i)$ included in $X$. Finally, we say that $\mu$
is \emph{S1 on a partial type $\pi(x)$} if $\pi(x)$ is $\mu$-wide and
included in a definable set on which $\mu$ is S1.
\end{defi}

The following results all appear in \cite{Hru12}.

\begin{fait}\label{HrStable}
Let $p$ be a type and assume that $\mu$ has the S1 property. Then for any type $q$ the relation
\[
R(a,b) \iff a^{-1}p(x) \cap b^{-1}q(x) \text{ is $\mu$-wide},
\]
where we identify a type with its realizations in the monster model, is a stable relation.
\end{fait}

\begin{fait}\label{wide}
Let $\mu$ be an $M$-invariant ideal which is S1 on some set $X$.
Then for any type $p(x)$ whose realizations are contained in $X$,
if $p(x)$ is $\mu$-wide, then $p(x)$ does not fork over $M$.
\end{fait}

Finally, the following  is Lemma 2.3 in \cite{Hru12}.

\begin{fait}
Let $p,q$ be complete types over a model $M$ and let $R(x,y)$ be a stable $M$-invariant relation in the realizations of $p(x)\times q(y)$. Then the truth value of $R(a,b)$ is constant
for all $a\models p(x)$ and $b\models q(y)$ as long as either $\tp(a/Mb)$ or $\tp(b/Ma)$ does not fork over $M$.
\end{fait}

The usefulness of S1 ideals can be seen in the following
proposition (which we will use in Section \ref{ShelahExpPRC}). It
is basically Lemma 6.1 in \cite{HrPi}, but the contexts are
not exactly the same.

\begin{defi}
A \emph{definable ideal}, is defined to be an ideal $\mu$ such
that for any $\phi(x;y)$, the set $\{b : \phi(x;b)\in \mu\}$ is
definable.
\end{defi}

\begin{prop}\label{TdefGps}
Let $G$ be a definable group equipped with a definable (left-)
$G$-invariant S1 ideal $\mu$. Let $H\leq G$ be a type-definable
subgroup of $G$ which is $\mu$-wide, then $H$ is the intersection
of definable subgroups of $G$.
\end{prop}
\begin{proof}

Write $H=\bigcap_{n<\omega} H_n$, where each $H_n$ is definable,
stable under inverse and $H_{n+1} \cdot H_{n+1} \subseteq H_n$.
Let $\delta_n(x;y)= x\in G \wedge y\in G \wedge xH_n \cap yH_n
\notin \mu$. Then $\delta_n$ is a definable, stable (as $\mu$ is
S1), $G$-invariant relation. Let $S_{\delta_n,H}$ be the set of
global $\delta_n$-types (in variable $x$) which are consistent
with $H$. By stability, all $\delta_n$-types are definable. Recall
that an element of $S_{\delta_n,H}$ is generic if every set in it
covers $G$ in finitely many translates. By Lemma 5.16 in
\cite{HrPi}, there are finitely many generic types in
$S_{\delta_n,H}$.

Let $Q_n$ be the definable set $\{b\in H_0 : \delta_n(x;b)$ is in
all generic types of $S_{\delta_n,H}\}$.

\begin{claim} $H= \bigcap Q_n$.
\end{claim}

\begin{claimproof} Let $a\in H$, then for $b\in H$ realizing a generic type of
$S_{\delta_n,H}$ over $a$, $H\subseteq aH_n \cap bH_n$, hence
$aH_n \cap bH_n\notin \mu$ and $a\in Q_n$.

Conversely, let $a\in \bigcap Q_n$ and take $b\in H$ generic over
$a$ as above. By definition of $Q_n$, we have $aH_n \cap bH_n
\notin \mu$ for all $n$ so that in particular it is non-empty.
Hence by compactness, $aH\cap bH$ is non-empty, so $a\in H$.
\end{claimproof}

\begin{claim} $HQ_n \subseteq Q_n$.
\end{claim}

\begin{claimproof} Let $a\in H$ and $b\in Q_n$. Let $c$ be generic over $a,b$. We
need to show that $abH_n \cap cH_n\notin \mu$. By invariance, this
is equivalent to $bH_n \cap a^{-1}cH_n\notin \mu$. But $a^{-1}c$
realizes a generic over $b$, hence this follows from the fact that
$b\in Q_n$.
\end{claimproof}

Finally, let $G_n = \{x \in H_n : xQ_n \subseteq Q_n \wedge
x^{-1}Q_n \subseteq Q_n\}$. Then $G_n$ is a subgroup and
$H\subseteq G_n \subseteq H_n$, so $H=\bigcap G_n$.
\end{proof}

\subsection{Stabilizer Theorems}

A good insight for invariant and S1 ideals (which prompted some of
the terminology we use) comes from measures. If we have a finitely
additive measure on definable sets which is invariant under
automorphisms fixing some set $A$, a natural $A$-invariant ideal
is that of sets of measure 0. In this context, \emph{wide} sets
are those which have positive measure. This ideal does not need to
be S1 since an infinite union of positive measure sets need not
intersect if the ambient universe has infinite measure. However,
if we restrict ourselves to a finite measure set, then the ideal
of measure 0 sets is in fact S1.

Because of this analogy, given an ideal $\mu$ we will call a definable set $X$ \emph{medium}
if  $\mu$ is S1 when restricted to $X$. Note that medium sets form an ideal.
A type is medium if it concentrates on a medium set. If
$p$ is medium and $a\models p$ with
$\tp(a/Mb)$ wide, then $\tp(a/Mb)$ does not fork over $M$.

\bigskip

Recall that we assume the ideal $\mu$
to be both $A$-invariant and invariant under translations by elements of $G$.
If $q$ and $r$ are wide types, then we define $St(q,r):=\{g : gp
\cap r$ is wide$\}$. If $p$ is wide, we will denote $St(p,p)$ by
$St(p)$ and $St_r(p)= \{g : pg\cap p$ is wide$\}$. Hence $g\in
St(p)$ if and only if there is some $a\models p$, $\tp(a/Mg)$ wide
and $ga\models p$ (then also $\tp(ga/Mg)$ is wide by
$G$-invariance of $\mu$). Observe that $St(p)$ is stable under
inversion. Finally, $Stab(p)$ is the subgroup generated by
$St(p)$.

If $p$ and $q$ are two types, we let $p\nf q=\{(a,b) : a\models p, b\models q, \tp(b/Ma)$ does not fork over $M\}$.

\smallskip
We recall one version of Hrushovski's stabilizer theorem from
\cite{Hru12}.

\begin{fact}[\cite{Hru12}]\label{fact_stab}
Let $\mu$ be an $M$-invariant ideal on $G$ stable under left and right multiplication. Let $X\subseteq G$ be a symmetric $M$-definable set such that $\mu$ is S1 on $X^3$. Let $q$ be a wide type over $M$ concentrating on $X$. Assume

\begin{description}
\item[(F)] There are $a,b \models q$ such that $\tp(a/Mb)$ and $\tp(b/Ma)$ are both non-forking over $M$.
\end{description}
Then there is a wide type-definable subgroup $S$ of $G$. We have $S=(q^{-1}q)^2$ and $qq^{-1}q$ is a coset of $S$. Moreover $S$ is normal in the group generated by $X$ and $S\setminus (q^{-1}q)$ is included in a union of non-wide $M$-definable sets.
\end{fact}

We will not actually use this theorem, but some modified versions of it, which we prove in this section. Theorem \ref{th_babystab} below is very close to Fact \ref{fact_stab}. The proof is of course very much inspired, at times literally copied, from that of Hrushovski. One difference is that we assume the ideal to be S1 on up to four products of the type and its inverse (instead of three), and this allows us to simplify slightly the arguments. On the other hand, we weaken the requirements by dropping assumption (F) and under assumption (B1), we forgo right-invariance.

The proof in \cite{Hru12} operates by acting on the right on
$q$, we decide to act on the left, which explains some differences
in the statements.

\medskip

We will need a stronger version of Fact \ref{HrStable}, where we restrict the requirement that $\mu$ has the S1 property in all sets.

\begin{lemme}
Let $p, q$ be medium, then the relation $R(g,h)$ defined as ``$gp\cap hq$ is wide" is a stable relation.
\end{lemme}
\begin{proof}
Note that by invariance of $\mu$, every translate of $p$ and $q$ is medium. Let $(g_i h_i :i \in \mathbb Z)$ be an indiscernible sequence and assume that $R(g_i,h_j)$ holds if and only if $i\leq j$.

\medskip \noindent
\underline{Case 1}: $g_0 p \cap g_1 p \cap h_2 q$ is wide.

\smallskip
We then have that for all $i>0$, $g_0 p \cap g_i p \cap h_{i+1} q$ is wide by indiscernibility. Also for $i<j$, we have $(g_i p \cap h_{i+1}q) \cap( g_{i+2} p \cap h_{i+3}q)$ is not wide as already $h_{i+1}q \cap g_{i+2}p$ is not wide. Therefore the sequence $(g_0 p \cap (g_{2i}p \cap h_{2i+1}q):i>1)$ contradicts the S1 property inside $g_0 p$.

\medskip \noindent
\underline{Case 2}: $g_0 p \cap g_1 p \cap h_2 q$ is not wide.

\smallskip
We know that for all $i<2$, $g_i p \cap h_2 q$ is wide. Hence the sequence $(h_2 q \cap g_i p : i<2)$ contradicts the S1 property inside $h_2 q$.
\end{proof}

\begin{lemme}\label{lem0}
Let $q,r$ be medium and wide, and let $p\in St(q,r)$. Take $(a,b)\in p\nf p$, then $a^{-1}b, b^{-1}a \in St(q)$.
\end{lemme}
\begin{proof}
Take $(a,b)\in p\nf p$. Since $St(q)$ is stable under inverses, it suffices to show that $a^{-1}b q\cap q$ is wide, which is equivalent to $b q \cap a q$ is wide. As $q$ is medium, by stability it is enough to prove this for one pair $(a,b)\in p\nf p$. Take $(a_i:i<\omega)$ an indiscernible sequence in $p$ such that $\tp(a_1/Ma_0)$ is non-forking over $M$. Then $a_i q \cap r$ is wide for all $i$, as $p\in St(q,r)$. As $r$ is medium, it follows that $a_0 q \cap a_1 q \cap r$ is wide. In particular $a_0 q \cap a_1 q$ is wide, as required.
\end{proof}

\begin{lemme}\label{lem1}
Let $p$ be wide and medium, and let $q\in St(p)$, take $(a,b)\models q\nf q$, then $a^{-1}b, b^{-1}a \in St(p)$. If $\mu$ is right invariant and if $q\in St_r(p)$, then $ab^{-1}, ba^{-1}\in St_r(p)$.
\end{lemme}
\begin{proof}
The first part follows from the previous lemma by taking $q,r$ there to be $p$ here. The second part of the statement is proved in the same way by multiplying on the right.
\end{proof}

We will also show the following.

\begin{lemme}\label{repeatedClaim}
Let $p$ and $r$ be medium types with $r$ wide, let
$(a,b)\models p \times_{nf} p$, and assume that $p^{-1}r$ is
medium. Then $ba^{-1}\in St(r)$.
\end{lemme}

\begin{proof}
We need to show that $a^{-1} r \cap b^{-1} r$ is wide. Let
$(a_i)_{i<\omega}$ be an indiscernible sequence of realizations of
$p$ such that $\tp(a_n/Ma_{<n})$ is non forking for all $n$. By
stability, it is enough to show that $a_0^{-1} r \cap a_1^{-1} r$
is wide. The type-definable sets $(a_i^{-1} r)_{i<\omega}$ are
wide and included in $p^{-1}r$ which is medium by hypothesis, so
by the S1 property $a_0^{-1} r\cap a_1^{-1} r$ is wide as
required.
\end{proof}

\begin{thm}\label{th_babystab}
Let $\mu$ be an $M$-invariant ideal on $G$ stable under left multiplication. Let $p\in S_G(M)$ be wide. Assume either (B1) or (B2), where:
\begin{description}
\item[(B1)] For some symmetric definable set $X\in p$, $\mu$ is S1 on $X^4$;

\item[(B2)] $\mu$ is S1 on $(pp^{-1})^2$ and invariant under (left and) right multiplication.

%\item[(F)] there are $a,b \models p$ such that $\tp(a/Mb)$ and $\tp(b/Ma)$ are both non-forking over $M$.

\end{description}
Then $Stab(p)=St(p)^2 = (pp^{-1})^2$ is a connected, wide type-definable group on which $\mu$ is S1. Furthermore $Stab(p)\setminus St(p)$ is included in a union of non-wide $M$-definable sets.
\end{thm}
\begin{proof}
Note that under either of (B1) or (B2), we have that both $p$ and $p^{-1}p$ are medium.

The proof will proceed by a series of steps. Only in the beginning will there be differences depending on whether (B1) or (B2) is assumed.

\smallskip \noindent
\begin{claim1}
Let $(a,b)\in p \nf p$, then $ba^{-1}\in St(p)$.
\end{claim1}
\begin{claimproof}
 This follows immediately from Lemma \ref{repeatedClaim}.\end{claimproof}

\begin{claim1'}
If (B1) holds, then for any $(a,b)\in p \nf p$ we have $a^{-1}b\in St(p)$.
\end{claim1'}
\begin{claimproof}
By symmetry of $X$, we have that $p^2$ is medium, so the result
follows from Lemma \ref{repeatedClaim} with $p=p^{-1}$ and
$r=p$.\end{claimproof}

\medskip

Take now $(a,b)\in p\nf p$, $\tp(b/Ma)$ wide. We define $q=\tp(a^{-1}b/M)$ under assumption (B1) and $q=\tp(ba^{-1}/M)$ under assumption (B2). Then in both cases $q\in St(p)$, $q$ is wide (using right-invariance in the (B2) case) and medium. Notice that under either assumption $p^{-1}q$ is medium: under (B1) $p^{-1}q\subseteq X^3$ and under (B2) $p^{-1}q \subseteq p^{-1}pp^{-1}$.

So Lemma \ref{repeatedClaim} implies

\begin{claim2}
Let $(a,b)\in p \nf p$, then $ba^{-1}\in St(q)$.
\end{claim2}

\begin{claim3}
 Let $(b,c)\in Stab(q)\nf q$, then $bc\in
St(p)$. So in particular $Stab(q)\subseteq St(p)^2 \subseteq
(pp^{-1})^2$.
\end{claim3}

\begin{claimproof}
 As $St(q)$ is stable under inverse, we can write $b=b_1\cdots b_n$, with each $b_i \in St(q)$. We show the result by induction on $n$. For $n=0$, it follows from the fact that $q\in St(p)$.

Assume we know it for $n-1$ and take $b=b_1\cdots b_n$. We have to show that $b_n^{-1}\cdots b_1^{-1} p \cap cp$ is wide. As $b_n \in St(q)$, there is $c'\models q$, $\tp(c'/Mb_n)$ wide such that $b_n c' \models q$. We may also assume that $\tp(c'/Mb_0\ldots b_n)$ is wide. Then by translation invariance, $\tp(b_nc'/Mb_0\ldots b_n)$ is wide. By induction, $b_{n-1}^{-1}\cdots b_1^{-1} p \cap b_nc' p$ is wide, then so is $b_n^{-1}\cdots b_1^{-1}p \cap c'p$ and we conclude by stability.
\end{claimproof}

\begin{claim4}
Let $a,b\models p$, then $ab^{-1}\in St(q)^2$.

\end{claim4}
\begin{claimproof}
Take $c\models p$ such that $\tp(c/Mab)$ is non-forking over $M$. Write $ab^{-1} = (ac^{-1})(cb^{-1})$. By Claim 2 and the fact that $St(q)$ is closed under inverses, both $ac^{-1}$ and $cb^{-1}$ are in $St(q)$ and the claim follows.
\end{claimproof}

\begin{claim5}
$Stab(p)=Stab(q)=(pp^{-1})^2$ is wide and medium.
 \end{claim5}

 \begin{claimproof}
  By Claim 3, we have $Stab(q)\subseteq St(p)^2 \subseteq (pp^{-1})^2$. By Claim 4, $pp^{-1} \subseteq Stab(q)$ so also $(pp^{-1})^2 \subseteq Stab(q)$, hence $(pp^{-1})^2 = St(p)^2=Stab(q)$. Finally, since $Stab(q)$ is a subgroup, we have $Stab(p)=St(p)^2=Stab(q)$. By hypothesis $(pp^{-1})^2$ is medium, and it is wide since it contains $q$.
 \end{claimproof}

All that is left to prove is that $Stab(p)$ has no type-definable over $M$ proper subgroup of bounded index, and that any wide type in $Stab(p)$ lies in $St(p)$.

Let $T\leq Stab(p)$ be a type-definable over $M$ subgroup of bounded index. We have $pp^{-1}\subseteq Stab(p)$, hence for $a\models p$, $p\subseteq Stab(p)a$. So $p$ lies in a right coset $S_p$ of $Stab(p)$. This coset is $M$-invariant and hence type-definable over $M$. All  right cosets of $T$ in $S_p$ are type-definable over $M$ and as $p$ is a complete type over $M$, it must lie entirely within one of them. Therefore $pp^{-1}\subseteq T$ and $T=Stab(p)$.

Now, let $s$ be a wide type in $Stab(p)=Stab(q)$. For any $b\in
Stab(q)$ we have $b^{-1}\in Stab(q)$. By Claim 3, if $c\models q$
is such that $\tp(c/Mb^{-1})$ wide we have that $b^{-1}cp\cap p$
is wide, so by invariance $cp\cap bp$ is also wide. By stability,
the same holds assuming instead that $\tp(b/Mc)$ is wide. Let
$c\models q$ and $b\models s$ such that $\tp(b/cM)$ is wide. Then
by left invariance, $\tp(cb/cM)$ is wide. But we also have $cb \in
Stab(q)$, hence $cbp\cap cp$ is wide. From which it follows that
$bp \cap p$ is wide, so $s$ lies in $St(p)$, as
required.\end{proof}

The proofs of the following two propositions are taken essentially without change from \cite{Hru12}.

\begin{prop}\label{prop_normal}
Let $\mu$ be an $M$-invariant ideal on $G$ stable under left and
right multiplication. Let $p\in S_G(M)$ be wide. Assume also (B1).

Then $Stab(p)$ is normal and of bounded index in the group
generated by $X$ and $X^n$ is medium for all $n$.
\end{prop}
\begin{proof}
Write $S=Stab(p)$. Let $r$ be a type over $M$ of elements of $X$. Then the image of $r$ in $G/S$ is bounded. Indeed, assume not, then we can find an indiscernible sequence $(a_i:i<\omega)$ of realizations of $r$ such that the cosets $a_i S$ are pairwise disjoint. Hence so are the types $a_i pp^{-1}$ (as $pp^{-1}\subseteq S$), but this contradicts S1 inside $X^3$. As $r$ is a complete type over $M$ it must be included in one left coset of $S$. Applying the same reasoning to $r^{-1}$, we see that $r$ is also included in a unique right coset of $S$. Thus $X/S$ is bounded and if $c,c'\models r$, then $c S c^{-1} = c' S c'^{-1} =: S^r$ is type-definable over $M$.

We now claim that $p^{-1}$ has bounded image in $G/S^r$: for if not, we would have an $Mc$-indiscernible sequence $(a_i:i<\omega)$ of realizations of $p$ with $a_i^{-1}c Sc^{-1}$ pairwise disjoint and again $a_i^{-1}cpp^{-1}$ would be pairwise disjoint contradicting S1 in $X^4$. Hence $p^{-1}$ lies entirely within one left coset of $S^r$ and $pp^{-1} \subseteq S^r$. Therefore $S\leq S^r$. We also have $S\leq S^{r^{-1}}$ and then $S=S^r$.

We have shown that $S$ is normalized by $X$ and has bounded index in it. It follows that $S$ has bounded index in any $X^n$, thus $X^n$ is medium.
\end{proof}

\begin{prop}
If we assume that both conditions (B1) and (B2) (equivalently (B1) and right-invariance) hold, then $pp^{-1}p$ is a coset of $Stab(p)$.
\end{prop}
\begin{proof}
Let $c\models p$. By the previous proposition $Stab(p)$ is normal in the group generated by $X$. Since $pp^{-1}\subseteq Stab(p)$, $p$ lies entirely within one coset of $Stab(p)$ and hence $pp^{-1}p\subseteq Stab(p)c$. Conversely, take any $a\in Stab(p)c$ and let $b\models p$ such that $\tp(b/Ma)$ is wide. Then $ba^{-1}\in Stab(p)$ and $\tp(ba^{-1}/M)$ is wide by right-invariance. By Theorem \ref{th_babystab} any wide type in $Stab(p)$ is in $St(p)$, so $ba^{-1}\in St(p) \subseteq pp^{-1}$. So $a=ab^{-1}b \in pp^{-1}p$.
\end{proof}

We will now prove a stabilizer theorem which changes the
hypothesis of the previous ones in a manner which is tailored to
prove Theorem \ref{algebraic group chunk}. Possibly the best way
to understand the strength and need for the new hypothesis
(compared for example with Fact \ref{fact_stab}) is to read the
proof of the Theorem \ref{algebraic group chunk} and the footnote
we added there.

The main change of the hypothesis consists of relaxing the
requirement that $\mu$ is S1 on $(pp^{-1})^2$ and assume only that
$\mu$ is S1 on \emph{generic} products in $p^{-1}p$ (see condition
(B) below). As mentioned before, the need for this will be clear
in the proof of \ref{algebraic group chunk}, where we cannot
require S1 in all of $p^{-1}p$. We manage to achieve this at the
cost of introducing a technical assumption (A) for which we need
to introduce a second ideal $\lambda$ that will serve as a more
restrictive notion of medium. We will assume that $\lambda$ is
also invariant under left translations by elements of $G$. A type
which is not $\lambda$-wide will be called
\emph{$\lambda$-medium}. In Theorem \ref{algebraic group chunk},
this restriction will be key in order to show that Condition (A)
holds. It seems plausible that for many, or all, ideals $\mu$,
condition (A) holds with $\lambda$ being the ideal of all medium
sets. We were however not able to prove any general statement of
this kind.

\begin{thm}\label{th_stabilizer}
Let $\mu$ and $\lambda$ be $M$-invariant ideals on $G$ as above,
stable under left and right multiplication, and such that $\mu$ is
S1 in any $X\in \lambda$.

Assume we are given a wide and medium type $p$ in $G$ and the following conditions are satisfied:
\begin{description}
\item[(A)] for any types $q$, $r$, if for some $(c,d)\models q \nf
r$, $\tp(cd/M)$ or $\tp(dc/M)$ is $\lambda$-medium, then $q$ is
$\lambda$-medium; \item[(B)] for any $(a,b)\in p \nf p$,
$\tp(a^{-1}b/M)$ is $\lambda$-medium; \item[(F)] there are
$(a,b)\models p \nf p$ such that $\tp(a/Mb)$ does not fork over
$M$.
\end{description}
Then $Stab(p)=St(p)^2 = (pp^{-1})^2$ is a connected
type-definable, wide and $\lambda$-medium group. Also
$Stab(p)\setminus St(p)$ is contained in a union of non-wide
$M$-definable sets.
\end{thm}
\begin{proof}
%It will follow from the proof that it is enough to assume that $pp^{-1}$ is medium, but that assumption will be too strong for us.
Throughout this proof, we will refer to $\lambda$-medium as
``medium''.

Condition (A) implies that if $q$ is a medium type, then both
$St(q)$ and $St_r(q)$ are medium. Together with Condition (B)
it also implies that $p^{-1}$ is medium.

\begin{claim1} If $(a,b)\in p \nf p$, then $ba^{-1}\in St(p)$.
\end{claim1}
\begin{claimproof}
We have to prove that $ba^{-1}p\cap p$, or equivalently $a^{-1}p\cap b^{-1}p$, is wide.
By stability, it is enough to prove this for some pair $(a,b)\in p\nf p$.
Let $(a_i:i<\omega)$ be an indiscernible sequence in $p$ such that $\tp(a_i/Ma_{<i})$
does not fork over $M$. Take $a=a_0$ and $b=a_1$.
Let $r:=tp(a^{-1}b/M)$, which is medium by
Condition (B). Also, since $tp(b/Ma)$ is wide and $\mu$ is left
invariant we know that $tp(a^{-1}b/Ma)$ is wide. Now,
$a^{-1}b\models a^{-1}p\cap r$ so $a^{-1}p\cap r$ is wide. Since
$a,b$ start an  indiscernible sequence, by S1 we have that
$a^{-1}p\cap b^{-1}p\cap r$ is wide, so $a^{-1}p\cap b^{-1}p$ is
wide as required.
\end{claimproof}

\begin{claim1'} If $(a,b)\in p \nf p$, then $a^{-1}b\in St_r(p)$.
\end{claim1'}
\begin{claimproof} By Claim 1, we have that $ba^{-1} \in St(p)$, so
in particular $r':=tp(ba^{-1}/A)$ is medium. Now, as in the
previous claim using S1 and invariance we have that
$tp(ba^{-1}/Aa)$ is wide. Since $ba^{-1}$ realizes $pa^{-1}\cap
r'$ the latter must be wide, and by S1 $pa^{-1}\cap pb^{-1}$ is
wide. By invariance $pa^{-1}b\cap p$ is wide, as required.
\end{claimproof}

Let $\mu'$ be the ideal defined by $\phi(x)\in \mu' \iff
\phi(x^{-1})\in \mu$. Then $\mu'$ is $M$-invariant, invariant
under left and right multiplication and is S1 on any inverse of a
medium type. We will write $St'$, $Stab'$ for the stabilizers with
respect to $\mu'$. Notice that since $p^{-1}$ is wide and medium,
$p$ is $\mu'$-wide and $\mu'$ is S1 on $p$.

Let $(a,b)\models p\times p$, $\tp(b/Ma)$ wide (hence non-forking
over $M$) and $q=\tp(ab^{-1}/M)$. Then $q$ is $\mu'$-wide and is
in $St(p)$, as $St(p)$ is closed under inverses, and thus $q$ and
$q^{-1}$ are medium. Also if $(c,d)\models q\nf q$, then
$\tp(c^{-1} d/M) \in St(p)$ by Lemma \ref{lem1}. In particular
$\tp(c^{-1}d/M)$ is medium.

\begin{claim2} If $(b,c)\in Stab'(q)\nf q$, then $bc\in St(p)$.
\end{claim2}
\begin{claimproof}As $St'(q)$ is stable under inverse, we can write $b=b_1\cdots b_n$, with each $b_i \in St'(q)$. We show the result by induction on $n$. For $n=0$, it is clear.

Assume we know it for $n-1$ and take $b=b_1\cdots b_n$. We have to show that $b_n^{-1}\cdots b_1^{-1} p \cap cp$ is wide. As $b_n \in St'(q)$, there is $c'\models q$, $\tp(c'/Mb_n)$ $\mu'$-wide such that $b_n c' \models q$. We may also assume that $\tp(c'/Mb_1\ldots b_n)$ is $\mu'$-wide. Then by translation invariance, $\tp(b_nc'/Mb_1\ldots b_n)$ is $\mu'$-wide. By induction, $b_{n-1}^{-1}\cdots b_1^{-1} p \cap b_nc' p$ is wide. We conclude by stability.
\end{claimproof}

\begin{claim3} There is $(a,b)\models p\nf q$, $\tp(b/Ma)$ $\mu'$-wide, such that $\tp(a^{-1}b/M)$ and its inverse are medium.
\end{claim3}

\begin{claimproof} By (F) there is $(c,d)\in p\nf p$ such that also $\tp(c/Md)$ does not fork over $M$. Let $r=\tp(d^{-1}c/M)$.
Let $a\models p$ and choose $b_0$ such that $\tp(a,b_0/M)=\tp(d, c/M)$. Then $a^{-1} b_0\models r$ and $\tp(b_0/Ma)$ does not fork over $M$.
Now choose $b_1\models p$ such that  $\tp(b_1/Mb_0)$ is wide and $\tp(b_0b_1^{-1}/M)=q$. We can furthermore assume that $\tp(b_1/Mab_0)$ is wide. By translation invariance, $\tp(b_0b_1^{-1}/Ma)$ is $\mu'$-wide. Now pick $b_2$ such that $\tp(b_2/Mab_0b_1)$ is non-forking over $M$ and $\tp(b_1,b_2/M) = \tp(c,d/M)$ so that $\tp(b_1^{-1}b_2/M)= r^{-1}$. By transitivity of non-forking, we have $\tp(b_1^{-1}b_2/Mab_0)$ is non-forking over $M$. Hence $(a^{-1}b_0,b_1^{-1}b_2)\models r \nf r^{-1}$.

By Claim 1' and since $St_r(p)$ is stable under inversion, $r\in
St_r(p)$ and by Lemma \ref{lem1}, $a^{-1}b_0b_1^{-1}b_2$ is also
in $St_r(p)$. It follows that $\tp(a^{-1}b_0b_1^{-1}b_2/M)$ and
its inverse are medium. By hypothesis (A),
$\tp(a^{-1}b_0b_1^{-1}/M)$ and its inverse are medium.
\end{claimproof}

\begin{claim4} If $(a,a')\models p\nf p$, then $aa'^{-1}\in
St'(q)$.\end{claim4}

\begin{claimproof} By Claim 3, (and because $p$ is $\mu'$-wide), we can
find $b\models q$ and some $a_0\models p$ with $\tp(b/Ma_0)$
$\mu'$-wide $r= \tp(a_0^{-1}b/M)$ and its inverse are medium.

Extending we can find a sequence $( a_i)_{i\in
\kappa}$ such that $\tp(b/Ma_{<\kappa})$ is non forking and, since
$\mu'$ is medium in $q$, $\tp(b/Ma_{<\kappa})$ is $\mu'$-wide. By
Erd\H os-Rado  if we take $\kappa$ large enough we can find a
subsequence $(a'_i)_{i<\omega}$ indiscernible over $Mb$.

Now, $a_0'^{-1}b \in a_0'^{-1}q\cap r$. By translation invariance,
$\tp(a_0'^{-1}b/Ma_0')$ is $\mu'$-wide, hence $a_0'^{-1}q \cap r$
is $\mu'$-wide. By indiscernibility, $a_i'^{-1}q \cap r$ is
$\mu'$-wide for all $i$. As $\mu'$ is S1 on $r^{-1}$, it follows
that $a_0'^{-1}q\cap a_1'^{-1}q$ is $\mu'$-wide.

The claim follows by stability.
\end{claimproof}

\medskip
Now we can conclude: we have, by Claim 2, $Stab'(q)\subseteq
St(p)^2 \subseteq (pp^{-1})^2$. Let $a,b \models p$ and choose
$c\models p$ such that $\tp(b/Mc)$ and $\tp(c/Mb)$ do not fork
over $M$ (using (F)). We can furthermore assume that $\tp(c/Mab)$
does not fork over $M$. Then $(a,c)\models p\nf p$ and
$(c,b)\models p\nf p$ and $ab^{-1} = (ac^{-1})(cb^{-1})$. By Claim
4, both $ac^{-1}$ and $cb^{-1}$ are in $St'(q)$, therefore
$ab^{-1} \in Stab'(q)$. We thus have $pp^{-1}\subseteq Stab'(q)$.
Therefore $Stab'(q)=St(p)^2=(pp^{-1})^2$ and as $Stab'(q)$ is a
subgroup, $Stab'(q)=Stab(p)$. Type-definability of $Stab(p)$ is
clear, so is wideness. The fact that $Stab(p)=Stab'(q)$ is medium
follows from Claim 2 and property (A).

Connectedness is proved as in of Theorem \ref{th_babystab}.
Finally, the fact that any wide type in $Stab(p)$ lies in $St(p)$
is proved as in Theorem \ref{th_babystab} replacing $Stab(q)$
there by $Stab'(q)$.
%It remains to see that any type $s$ in $Stab(p)$ which is wide is in $St(p)$. This is exactly the same argument as in Theorem \ref{th_stabilizer}, but we repeat it here for convenience.
%
%so let $s$ be a wide type in $Stab(p)=Stab'(q)$. By Claim 2, for any $b\in Stab'(q)$ and $c\models q$ with $\tp(c/Mb)$ $\mu'$-wide, $b^{-1}p\cap cp$ is wide, hence so is $bp\cap cp$ since also $b^{-1}\in Stab'(q)$. By stability, the same holds assuming instead that $\tp(b/Mc)$ is wide.
%Let $c\models q$ and $b\models s$ such that $\tp(b/cM)$ is wide. Then by left invariance, $\tp(cb/cM)$ is wide. But we also have $cb \in Stab'(q)$, hence $cbp\cap cp$ is wide. From which it follows that $bp \cap p$ is wide as required.
\end{proof}

The following lemma will be useful later to check that the
hypotheses of Theorem \ref{algebraic group chunk} are satisfied.

\begin{lemme}\label{lem_stqr}
Assume that $\mu$ is left invariant and condition (A) holds. Let $q$, $r$ be medium and wide types. Let $p\in St(q,r)$ be a wide type and take $(a,b)\in p \nf p$. Then $\tp(a^{-1}b/M)$ is medium.
\end{lemme}
\begin{proof}
We show that $a^{-1}b\in St(q)$, {\it i.e.}, that $aq \cap bq$ is
wide. As $q$ is medium, by stability, it is enough to show this
for some pair $(a,b)\in p\nf p$. Take $(a_i:i<\omega)$ an
indiscernible sequence in $p$ with $\tp(a_n/Ma_{<n})$ wide; it is
enough to show that $a_0q \cap a_1q$ is wide. By assumption $a_0 q
\cap r$ is wide. As $r$ is medium, by the S1 property, $a_0q \cap
a_1q \cap r$ is wide, hence $a_0q \cap a_1 q$ is wide as required.
\end{proof}

\subsection{Applying the Stabilizer Theorem: algebraic group chunks}

This section is devoted to proving that Theorem
\ref{th_stabilizer} implies the existence of large algebraic
subgroups in many theories, which can be seen as a generalization
of results in \cite{HrPi}.

We will need to adapt some of the definitions from \cite{HrPi}.

\begin{defi}
A theory $T$ in a language containing the language of rings and
which contains the theory of fields, is \emph{algebraically
bounded} if, given any formula $\phi(\bar{x},y)$, there are
polynomials $f_1(\bar{x}, y),\ldots, f_n(\bar{x},y) \in
\mathbb{Z}[\bar{x},y]$ such that, whenever $K$ is a model of $T$
and $\bar{a}$ is a tuple of elements of $K$ such that
$\phi(\bar{a},K):= \{y \in K: \phi(\bar{a},y)\}$ is finite, then
there is an index $i$ such that the polynomial $f_i(\bar{a},y)$ is
not identically $0$ on $K$ and $\phi(\bar{a},K)$ is contained in
the set of roots of $f_{i}(\bar{a},y)= 0$.
\end{defi}

The following is Theorem 3.1 in \cite{HrPi}, which can be seen as
an ``algebraic group configuration'' theorem.

\begin{fait}\label{fact_groupconf_original}
Let $T$ be a theory extending the theory of fields
which is algebraically bounded.
Let $\mathcal{U}$ be a monster model of $T$.
Let $G$ be a group definable in $T$
over a set $A$, and let $a,b,c\in G(\mathcal{U})$ be such that $a\cdot_G b=c$ and such that
$a$ and $b$ are algebraically independent over $A$.

Then there is a set $B$ containing $A$ such that $a$ and $b$ are still algebraically independent
over $B$, a $B$-definable algebraic group
$H$ and dimension-generic elements $a', b', c'\in H(\mathcal{U})$
such that $a'\cdot b' = c'$ and $\acl(Ba)=\acl(Ba')$,
$\acl(Bb)=\acl(Bb')$ and $\acl(Bc)=\acl(Bc')$.
\end{fait}

We will prove the following, ``algebraic group chunk'' theorem.

\begin{thm}\label{algebraic group chunk}
Let $T$ be a theory extending the theory of fields
which is algebraically bounded and such that any model of $T$ is
definably closed in its algebraic closure. Let $G$ be a group
definable in a $\omega$-saturated model $M$ of $T$. Assume that
$T$ admits an $M$-invariant ideal $\mu_G$ on $G$, stable under
left and right multiplication, and such that $\mu_G$ is S1 in $G$.
Finally, assume also that there is a $\mu_G$-wide type $p$ such
that condition (F) holds: There are $(a,b)\models p \nf p$ such
that $\tp(a/Mb)$ does not fork over $M$.

Then there is an algebraic group $H$ and a definable finite-to-one
group homomorphism from a type-definable wide subgroup $D$ of $G$
to $H(M)$.
\end{thm}

We begin with the following proposition.

\begin{prop}\label{prop_groupconf}
Let $T, \mu_G,p$ and $M$ be as in the statement of Theorem
\ref{algebraic group chunk}. Let $\mathcal{U}$ be a monster model of
$M$. Let $a\models p|M$, $b\models p|Ma$, and $c=a\cdot_G b$. Then
$\tp(c/Ma)$ is $\mu_G$-wide, and there is an $M$-definable
algebraic group $H$ and dimension-generic elements $a', b', c'\in
H(\mathcal{U})$ such that $a'\cdot b' = c'$ and
$\acl(Ma)=\acl(Ma')$, $\acl(Mb)=\acl(Mb')$ and
$\acl(Mc)=\acl(Mc')$.
\end{prop}

\begin{proof}
Let $A$ be the (finite) set of parameters over which $G$ is
defined.

Note that compared to Fact \ref{fact_groupconf_original}, we require that the set $B$
in the statement can be found inside $M$. This is clear
throughout the proof in \cite{HrPi}, except maybe for the last
base change. We will therefore recall the stage of the
construction prior to the last base changes, and show why we can
complete the proof with our requirements.

Let $M^{alg}$ be the field theoretic algebraic closure of $M$ in
the language of rings (so a model of algebraically closed fields).

The construction yields elements  $a_1, b_1, c_1$ in $\mathcal U$
satisfying the algebraic relations in the statement of the
theorem, and $\sigma$ the canonical base (in $M^{alg}$) of
$\tp(b_1, c_1/A a_1)$. This type is stationary, so $\sigma$ is
definable in $\mathcal U$. The element $\sigma$ defines a map from
$q_1:=\qftp(b_1/A)$ to $q_2:=\qftp(c_1/A)$, and any $b_2, c_2$
realizations of $q_1, q_2$ in $\mathcal U$, define some $\sigma'$
which (because $\mathcal U$ is definably closed in its algebraic
closure) will be in $\mathcal U$.

Take independent $\sigma_1, \sigma_2\models \tp(\sigma/A)$ and
elements $b_1', b_2'$ and $c'\models q_2$ with $\sigma_1(b_1')=c'$
and $\sigma_2(b_2')=c'$. If we take $\tau_{1,2}$ to be the
canonical base of $\tp(b_1', b_2'/A\sigma_1, \sigma_2)$, then one
can show (and it is shown in \cite{HrPi}) that $\tau_{1,2}$ gives
the germ of a function from $q_1$ to itself sending $b_1'$ to
$b_2'$ and thus can be identified with what would be the function
$\sigma_2^{-1}\circ \sigma_1$.

Notice that by stationarity, given \emph{any} $b_1', b_2'$
realizations of $q_1$ we can find some $c'$ realizing $q_2$ inside
$\mathcal U$. Then, if we take $\sigma_1'$ the canonical base of
$\tp(b'_1 c'/A)$ and $\sigma_2'$ the canonical base of $\tp(b'_2
c'/A)$, then the canonical base $\tau_{1,2}'$ of $\tp(b_1',
b_2'/A\sigma_1, \sigma_2)$ will be identified with a function
sending $b_1'$ to $b_2'$. Once again, if $b_1'$ and $b_2'$ were
chosen in $\mathcal U$, we get $\tau_{1,2}'$ in $\mathcal U$.

The proof in \cite{HrPi} now uses the stable group configuration
theorem (due to Hrushovski, stated as Proposition 1.8.1 in
\cite{HrPi}) which gives a $M^{alg}$-definable algebraic group $H$
with generic type $s^{qf}$ (the quantifier free formulas in $s$)
acting transitively on a set $X$ with generic type $q_1^{qf}$. So
$\tau$ is an element of $H(\mathcal U)$.

The proof then concludes by first adding $\sigma_1\models
\tp(\sigma/A)$ to the base (which can of course be chosen inside
$M$) and then choosing $\tau_1\models s$, define $b_2=\tau_1^{-1}
b_1$ and add $b_2$ to the base. In this order it is impossible to
guarantee that $b_2$ belongs to $M$. However, we can choose
$b_2\in M$ a realization of $q_1$, and choose $\tau_{2,1}$ be the
germ sending $b_2$ to $b_1$. As discussed above this can always be
chosen and $\tau_{2,1}$ would be an element of $H(\mathcal U)$.
\end{proof}

\begin{rem}
Barriga \cite{Ba} dealt with the choice of $b_2$ in a different
way in the context of bounded groups definable in real closed
fields. However, her proof does not work in the general
context we are working with (specifically, it requires
``rosiness'' of $T$).
\end{rem}

\begin{comment}
We go through the proof to see that this can be arranged. The
elements $a,b,c$ in the proof were picked at the beginning and
never changed in the proof, so we only need to verify that all the
base extensions done during the construction can take place inside
$M$. But this holds since whenever a new point is introduced, it
is chosen independently (in the sense of dimension) from all the
points chosen so far, and we can always do that inside $M$. More
precisely, following the notations of the proof: in Lemma 3.2, we
first chose $x'$ in $G(M)$, then pick $z_1 \in G(M)$ and this
ensures that $A_2 \subseteq M$. The set $A$ after the proof of 3.2
is equal to $\acl(A_2)\cap F$ and is therefore in $M$. There are
only two base extensions left, done at the very end. We chose
$\sigma_1 \in M$ and have $A_1 = \acl(A,\sigma_1)\cap F\subseteq
M$ (it seems that there is a typo in the paper, it says $A_1 =
\acl(A,\sigma)\cap F$ instead). Finally, take $\tau_1$ such that
$b_2 := \tau_1^{-1} \cdot b_1$ is in $M$. This can be ensured for
example by choosing $b_2$ first. Then $A_2 \subseteq M$ and we are
done.
\end{comment}

%This section is devoted to proving the following theorem.
\begin{proof}[Proof of Theorem \ref{algebraic group chunk}]

Let $a,b,c$ be as in the statement of Proposition \ref{prop_groupconf}.
So there is an $M$-definable algebraic group $(H, \cdot_H)$
and $a', b', c' \in H$ such that $c'= a'\cdot_H b'$, $\acl(Ma)=
\acl(Ma'), \acl(Mb)= \acl(Mb')$ and $\acl(Mc)= \acl(Mc')$.

We define an ideal $\mu$ on $G \times H$, by saying that $D\in
\mu$ if and only if $\pi_1(D)\in \mu_G$. Then $\mu$ is
$M$-invariant and invariant under left and right translations. We
will refer to $\mu$-wide as ``wide''.

We define the ideal $\lambda$ as the set of subsets $X$ of
$G\times H$ for which the projections to $G$ and $H$ each have
finite fibers. Thus $\lambda$ is included in the ideal of sets that are medium for $\mu$.
As before, a set in $\lambda$ will be called $\lambda$-medium.
Define $\widetilde{p}= \tp(a,a'/M)$. Then, because
$(a,a')$ is inter-algebraic with $a$ over $M$, $\widetilde{p}$ is
wide and medium and Condition (F) holds for $\widetilde{p}$.

We will show that conditions (A) and (B) of Theorem
\ref{th_stabilizer} also hold with the ideals $\mu$ and
$\lambda$\footnote{Notice that the product of two types with
finite fibers does not necessarily have finite fibers.
Hence we do not know that $pp^{-1}$ is medium.
This explains the need for the restrictive hypothesis on $p\times_{nf} p$ in Theorem \ref{th_stabilizer}}.

\begin{claim}
Condition (A) holds: If $p,q$ are two types in $G\times H$ and we
have $(g,h)\models p\times_{nf} q$ such that either $\tp(gh/M)$ or
$\tp(hg/M)$ is $\lambda$-medium, then $p$ is $\lambda$-medium.
\end{claim}

\begin{claimproof}
Denote $g=(g_0,g_1)$ and same for $h$. We will prove the case
where we assume that $\tp(gh/M)$ is $\lambda$-medium, the other case is
proved in an analogous way. Since $g_0h_0 \in \acl(Mg_1h_1)$ we
have $g_0 \in \acl(Mg_1h_0h_1)$. As $\tp(h_0h_1/Mg_0g_1)$ does not
fork over $M$, this implies that $g_0 \in \acl(Mg_1)$. In the same
way we get $g_1 \in \acl(Mg_0)$.
 \end{claimproof}

By Lemma \ref{lem_stqr}, condition (B) holds. We can then apply
Theorem \ref{th_stabilizer}, which gives us a connected,  medium,
wide type-definable group $K\leq G\times H$. As $K$ is
$\lambda$-medium, its projections to $G$ and $H$ have finite
fibers. As $K$ is wide, $\pi_1(K)$ is $\mu_G$-wide.

It only remains to show that we may assume that $\pi_1$ is
injective on $K$.

Let $K_1 = \pi_1^{-1}(e)\cap K$. Then $K_1$ is finite and normal
in $K$. As $K$ is connected, $K_1$ is central in $K$ (the
centralizer of $K_1$ is a relatively definable subgroup of $K$ of
finite index). Let $C\leq H$ be the centralizer of $\pi_2(K_1)$
inside $H$. It is an algebraic subgroup of $H$. Then we can
replace $H$ by $C/\pi_2(K_1)$ which is again an algebraic group
(defined over the same parameters as $H$ and $K_1$). Thus we may
assume that $K_1$ is trivial. % Hence from now on, we assume that $\pi_1$ and $\pi_2$ are injective on $K$.

This completes the proof of the theorem.
\end{proof}

An easy corollary of the theorem is the following result which we
believe was not known.

\begin{cor}
Let $R$ be a real closed field and let $G$ be a torsion free
definable group in $R$. Then $G$ is definably isomorphic to a
definable subgroup of an algebraic group.
\end{cor}

\begin{proof}
Any torsion free definable group definable in an o-minimal structure is solvable, so it is amenable as
a discrete group, and therefore definably amenable. By results in
\cite{CS} we know that $G$ admits a bi-f-generic type, and if we
define $\mu_G$ as the ideal of formulas which do not extend to
bi-f-generic types, then $\mu_G$ is $M$-invariant for some model
$M$, stable under left and right multiplication. Futhermore
$\mu_G$ is S1 in $G$ and any wide subgroup must contain $G^{00}$
(see Definition \ref{defG00}). We will reprove those facts in the
more general context of NTP$_2$ theories in Section \ref{SGroups}.

Condition (F) holds in any dependent theory, so in particular it
holds for real closed fields.

By Theorem \ref{algebraic group chunk} there is an algebraic group
$H$ and a definable finite-to-one group homomorphism $f$ from a
type definable wide subgroup $D$ of $G$ containing $G^{00}$ to
$H(M)$. But in torsion free groups definable in real closed fields
$G=G^{00}$. It follows by compactness that $D$ can be taken to be
definable. Finally, $ker(f)$ is a finite subgroup of the torsion
free $G$, so $ker(f)=\{e_G\}$ and $f$ is a definable injection, as
required.
\end{proof}

\section{Groups with f-generics in \ntp}\label{SGroups}

In this section we will use Theorem \ref{th_babystab} to prove Theorem \ref{th_stab}, which
is a stabilizer theorem for strong f-generic types in a group $G$
definable in an \ntp
 theory (see Definition \ref{strong f generics}).

 We work here with a complete theory $T$ and let $\mathcal U$ denote a monster model of $T$.

 We recall the definition of NTP$_2$.

\begin{defi}\label{NTP2}
We say that  $\phi(\bar{x}, \bar{y})$ has $TP_2$ if there are $(a_{l j})_{l, j < \omega}$ in $\mathcal U$ and $k \in \omega$ such that:
\begin{enumerate}
    \item $\{\phi(\bar{x}, a_{l, j})_{j \in \omega}\}$ is $k$-inconsistent for all $l< \omega .$
    \item For all $f:\omega \rightarrow \omega, \{\phi(\bar{x}, a_{l, f(l)}): l \in \omega\}$ is consistent.
\end{enumerate}
 A formula $\phi(\bar{x}, \bar{y})$ is \emph{$NTP_2$}  if it does not have $TP_2$.
The theory $T$ is \emph{$NTP_2$} if no formula has $TP_2$.

\end{defi}

We will assume throughout this section that $T$ is NTP$_2$. Let $G$ be a $\emptyset$-definable group.
Recall that an \emph{extension base} is a set $A$ such that no $p\in S(A)$ forks over $A$.
We will use the following results (the first three are from
\cite{ChKa} and the fourth one from \cite{BYC}).

\begin{fact}\label{fact_ntp}
Let $T$ be an NTP$_2$ theory and $A$ an extension base.

\begin{enumerate}
\item For any $b$, there is an $A$-indiscernible sequence
$(b_i:i<\omega)$ in $\tp (b/A)$ such that for any formula
$\phi(x;b)$ which divides over $A$, the partial type
$\{\phi(x;b_i):i<\omega\}$ is inconsistent.

\item A formula forks over $A$ if and only if it divides over $A$.

\item Condition $\mathrm{\mathbf{(F)}}$ is satisfied: given any
type $p$ over $A$, there are $a,b\models p$ such that $\tp(a/Ab)$
and $\tp(b/Aa)$ are non-forking over $A$.

\item The ideal of formulas which do not fork over $A$ is has the S1 property.

\end{enumerate}
\end{fact}

\begin{defi}\label{strong f generics}
A global type $p\in S_G(\monster)$ is \emph{strongly (left) f-generic
over $A$} if for all $g\in G(\monster)$, $g\cdot p$ does not fork
over $A$.

It is \emph{strongly bi-f-generic} if for all $g,h\in G(\monster)$, $g\cdot p \cdot h$ does not fork over $A$.
\end{defi}

It is proved in \cite{NIP2} that a definable group in an NIP theory is definably amenable (that is, admits a definable $G$-invariant measure on definable sets) if and only if it admits a strong f-generic type over some model. The theory of definably amenable NIP groups was studied in \cite{NIP1}, \cite{NIP2} and \cite{CS} (amongst other papers). In particular, the paper \cite{CS} characterizes in various ways formulas which extend to strong f-generic types. We generalize here those results to the NTP$_2$ context, assuming that $G$ admits a strong f-generic type. The proofs are very similar to those in \cite{CS}.

%We will now show that in any group $G$ with strong f-generics, the
%ideal $\mu_A$ of sets which are non f-generic over $A$ satisfies
%Condition \textbf{(B1)} of Theorem \ref{th_babystab}, which will
%then allow us to prove the main result of the section.

First, we generalize Proposition 5.11 (i) of \cite{NIP2}, with essentially the same proof.

\begin{lemme}
If for some model $M$, $G$ admits a strongly f-generic type over $M$, then the same is true over any extension base $A$.
\end{lemme}
\begin{proof}
We expand the structure by adding a new sort $S$ which, as a set,
is a copy of the group $G$ and we put all $G$-invariant relations
on it. So $S$ becomes a homogeneous space for $G$ and any point of
$S$ gives rise to a definable bijection between $S$ and $G$. This
expanded structure is $NTP_2$, and is conservative: it does not
add any definable sets to the main sort. Given any $A\subseteq
\monster$, there is a strongly f-generic type over $A$ if and only
if the formula $x_S = x_S$ in the expanded structure does not fork
over $A$. (See \cite[Proposition 5.11]{NIP2} or \cite[Lemma
8.19]{NIPbook}.)

Now assume that $x_S = x_S$ does not fork over some $M\subseteq \monster$ and let $A\subseteq \monster$ be an extension base. Let $\tilde N$ be an $|M|^+$-saturated model of the expanded theory containing $A$.

\begin{claim} In this expansion, the type $\tp(M/A)$ does not fork over $A$.
\end{claim}

\begin{claimproof} Assume it did. Then by definition it implies a disjunction of formulas, each dividing over $A$. As the expansion is conservative, we may assume that those formulas have parameters in the main sort. But then we can forget about the additional sort and use the fact that $\tp(M/A)$ does not fork over $A$ in the original structure as $A$ is an extension base.
\end{claimproof}

There is therefore $M'\equiv_A M$ such that $\tp(M'/\tilde N)$ does not fork over $A$. By assumption, there is some $d\in S$ such that $\tp(d/M'\tilde N)$ does not fork over $M'$. Then by transitivity of non-forking, $\tp(d/\tilde N)$ does not fork over $A$ as required.
\end{proof}

%Let $\mu_l(M)$ be the family of all formulas (over $\monster$), no left-translate of which forks over $M$. Define similarly $\mu_r(M)$ and $\mu_b(M)$. Hence an f-generic type over $M$ is a global type which is wide for $\mu_l(M)$. Note that those families are $M$-invariant and left (resp. right, bi)-$G$-invariant.

\begin{lemme}\label{lem_biexist}
Let $A\subseteq N$, where $N$ is $|A|^+$-saturated. Assume that
$p\in S(\monster)$ is strongly f-generic over $A$. Let $a\models
p|_N$ and $b\models p|_{N a}$. Then $\tp(ba^{-1}/N)$ extends to a
global type, strongly bi-f-generic over $A$.
\end{lemme}
\begin{proof}
Let $g,h\in G(N)$. Then $\tp(gb/Na)$ does not fork over $A$ and neither does $\tp(ha/N)$. By transitivity of non-forking, $\tp(gb,ha/N)$ does not fork over $A$. Hence $\tp(gba^{-1}h^{-1}/N)$ does not fork over $A$. Since $g,h$ were arbitrary in $G(N)$, this shows that $\tp(ba^{-1}/N)$ is strongly bi-f-generic over $A$.

Since $N$ is $|A|^+$-saturated, $\tp(ba^{-1}/N)$ extends to a global type strongly bi-f-generic over $A$. (This is a closed condition and any finite part of it can be dragged down into $N$.)
\end{proof}

We will say that the group \emph{$G$ has strong f-generics} if it has a strongly f-generic type over some/any extension base. By Lemma \ref{lem_biexist} it would then also have a strong bi-f-generic type over any extension base.

\begin{defi}
Let $\phi(x)\in L(A)$ be a formula. We say that \emph{$\phi(x)$ is f-generic over $A$} if no (left) translate of $\phi(x)$ forks over $A$. We say that \emph{$\phi(x)$ $G$-divides over $A$} if for some $A$-indiscernible sequence $(g_i:i<\omega)$ of elements of $G$, the partial type $\{g_i\cdot \phi(x):i<\omega\}$ is inconsistent.
\end{defi}

\begin{lemme}\label{lem_fund1}
Let $A$ be an extension base and $\phi(x)\in L(A)$. Then $\phi(x)$ is f-generic over $A$ if and only if it does not $G$-divide over $A$.
\end{lemme}
\begin{proof}
If for some $g\in G$, $\phi(g^{-1}x)$ forks over $A$, then it divides over $A$ and there is an $A$-indiscernible sequence $(g_i:i<\omega)$ such that $\{\phi(g_i^{-1} x):i<\omega\}$ is inconsistent. This shows that $\phi(x)$ $G$-divides over $A$. Conversely, if $\phi(x)$ $G$-divides over $A$ as witnessed by $(g_i:i<\omega)$, then $\phi(g_0^{-1}x)$ divides over $A$.
\end{proof}

Let $A\subseteq B$ be two extension bases over which $\phi(x;a)$ is defined.
Then $\phi(x;a)$ $G$-divides over $A$ if and only if it $G$-divides over $B$ so the same is true for f-generic. From now on, we drop the ``over $A$'' when talking about f-generic formulas.

\begin{lemme}\label{lem_genfork}
Assume that the formula $\phi(x;b)$ forks over $A$ and that $\tp(g/Ab)$ does not fork over $A$. Then $\phi(gx;b)$ forks over $A$.
\end{lemme}

\begin{proof}
Assume that $\phi(gx;b)$ does not fork over $A$ and let $c\models \phi(gx;b)$ with $c\ind_A Abg$. Then $c\ind_{Ag} Abg$. We also have $g \ind_A Ab$ by hypothesis. By transitivity, $gc \ind_A Ab$. Since $gc \models \phi(x;b)$, we get that $\phi(x;b)$ does not fork over $A$.
\end{proof}

\begin{prop}\label{prop_fund2}
Let $A$ be an extension base, $A\subseteq B$ and $\phi(x)\in L(B)$. Let $q$ be a global type strongly f-generic over $A$ and $g\models q|_B$. Then $\phi(x)$ extends to a global type strongly f-generic over $A$ if and only if $g^{-1}\cdot \phi(x)$ does not fork over $A$.
\end{prop}
\begin{proof}
Assume that $\phi(x)$ does not extend to a global type strongly f-generic over $A$. Then there are elements $g_i$, $i<n$ in $G(\monster)$ and formulas $\phi_i(x;b)\in L(\monster)$ each forking over $A$ such that $\phi(x) \vdash \bigvee_{i<n} \phi_i(g_i x; b)$. We can assume that $g$ realizes $q$ over $Bb\{g_i\}_{i<n}$. We have then that $\phi(gx)\vdash \bigvee_{i<n} \phi_i(g_ig x; b)$. Now, $\tp(g_ig/Ab)$ does not fork over $A$  for each $i<n$. By Lemma \ref{lem_genfork}, this implies that $\phi_i(g_ig x;b)$ forks over $A$. Hence $\phi(gx)=g^{-1}\cdot \phi(x)$ forks over $A$.

Conversely, if $\phi(x)$ extends to some global type strongly f-generic over $A$, then no translate of $\phi(x)$ forks over $A$ and in particular $g^{-1}\phi(x)$ does not fork over $A$.
\end{proof}

The previous results combine into the following equivalences.

\begin{prop}\label{prop_fund}
Let $A$ be an extension base and assume that there is a global
type $q$ strongly f-generic over $A$. Let $\phi(x)\in L(A)$ and
let $g$ realize $q$ over $A$. The following are equivalent:

1. $\phi(x)$ is f-generic;

2. $\phi(x)$ does not $G$-divide over $A$;

4. $g^{-1}\cdot \phi(x)$ does not fork over $A$;

5. $\phi(x)$ extends to a global type strongly f-generic over $A$.
\end{prop}

As usual, we extend definitions from definable sets to types: we
define a type to be \emph{f-generic} if it contains only f-generic
formulas. Notice that, because each definable subset of an
f-generic type may witness f-genericity in a different model, not
all f-generic types are strongly f-generic.

\begin{prop}\label{prop_fundweak}
Let $A$ be an extension base and assume that there is a global f-generic type $q$. Let $\phi(x)\in L(A)$ and let $g$ realize $q$ over $A$. Then $\phi(x)$ is f-generic if and only if $g^{-1}\cdot \phi(x)$ does not fork over $A$.
\end{prop}
\begin{proof}
If $\phi(x)$ is f-generic, then $g^{-1}\cdot \phi(x)$ does not fork over $A$ by definition.

Conversely, assume that $\phi(x)$ does $G$-divide and let $(g_i:i<\omega)$ be an $A$-indiscernible sequence witnessing it. Let $\hat q = q|_A$.

\begin{claim} The partial type $\bigcup g_i^{-1}\cdot \hat q$ is consistent.
\end{claim}

\begin{claimproof}
If not, then there is a formula $\psi(x)\in \hat q(x)$ such that $\{g_i^{-1}\cdot \psi(x):i<\omega\}$ is inconsistent. Then $g_0^{-1}\cdot \psi(x)$ divides over $A$, contradicting the assumption on $q$.
\end{claimproof}

Let $h$ realize $\bigcup g_i^{-1}\cdot \hat q$, so $g_i\cdot h \models \hat q$ for each $i$. Notice that $\{h^{-1}g_i^{-1} \cdot \phi(x):i<\omega\}$ is still $k$-inconsistent for some $k$, and $g^{-1}\cdot \phi(x)$ divides over $A$ as required.
\end{proof}

\begin{cor}
Assume that there is a global f-generic type, then the family $\mu$ of non-f-generic formulas is an ideal.
\end{cor}
\begin{proof}
Let $q$ be a global f-generic type. Let $\phi(x)$ and $\psi(x)$
be non-f-generic and take $M$ a model over which both are defined.
Let $g\models q|_M$ as in the previous proposition. Then
$g^{-1}\cdot \phi(x)$ and $g^{-1}\cdot \psi(x)$ both fork over
$M$, hence so does $g^{-1}\cdot (\phi(x)\vee \psi(x))$ --as
forking equals dividing over $M$-- which implies that $\phi(x)\vee
\psi(x)$ is not f-generic.
\end{proof}

\begin{question}
Assume that there is a global f-generic type; is there a strongly
f-generic type?
\end{question}

Notice that the ideal $\mu$ of non-f-generic formulas is $\emptyset$-invariant and invariant by translations on the left and on the right. It is however not S1 in general. For this we have to work with $\mu_A$.

Assume that $G$ has a strong f-generic type over $A$. Let $\mu_A$ be the ideal of formulas $\phi(x)\in L(\monster)$ which do not extend to a global type strongly f-generic over $A$. Then $\mu_A$ is $A$-invariant, left-$G$-invariant over $A$. By Proposition \ref{prop_fund}, $\mu$ and $\mu_A$ agree on $L(A)$.

\begin{lemme}\label{lem_s1}
The ideal $\mu_A$ is S1.
\end{lemme}
\begin{proof}
Assume that $(a_i:i<\kappa)$ is an $A$-indiscernible sequence such
that $\phi(x;a_i)$ extends to a type strongly f-generic over $A$.
Let $q$ be strongly f-generic over $A$ and let $g$ realize $q$
over $Aa_{<\kappa}$.  We can suppose that $\kappa$ is large
enough, then by Erd\H os Rado, there is a subsequence
$(a_{i_j})_{j<\omega}$ indiscernible over $Ag$. By Proposition
\ref{prop_fund} $g^{-1}\cdot \phi(x;a_{i_j})$ is non-forking over
$A$ for all $j$. As the non-forking ideal is S1 in \ntp~theories,
also $g^{-1}\cdot (\phi(x;a_{i_0})\wedge \phi(x;a_{i_1}))$ is
non-forking over $A$. By Proposition \ref{prop_fund},
$\phi(x;a_{i_0})\wedge \phi(x;a_{i_1})$ is $\mu_A$-wide.
\end{proof}

%Let $D$ be an $A$-definable set. As $\mu_A$ is S1, we know from \cite{hr_appr} that the relation $R(g,h)$ expressing that $gD \cap hD$ is $\mu_A$-wide is a stable $A$-invariant relation.

%If $p\in S(M)$ is $\mu$-wide, for some $M$-invariant ideal $\mu$, we let $St_\mu(p) = \{g \in G(\monster): g\cdot p \cap p$ is $\mu$-wide$\}$. Note that $St_\mu(p)\subseteq pp^{-1}$.

\subsection{Stabilizers of strong f-generic types}

We will need the following definitions.

\begin{defi}\label{defG00}
  Let $G$ be a definable group, and $M$ be a model over which $G$ is definable.

  We will say that a subset $X\subset G$ is \emph{generic} if finitely many translates cover $G$.

  If $H$ is a type definable (with parameters in $M$) subgroup of $G$ (or more generally an automorphism invariant subgroup),
  we will say that $H$ has \emph{bounded index in $G$} if
  we have that the cardinality of $G(M^*)/H(M^*)$ is smaller than the cardinality of $M^*$ for some saturated model $M^*$ extending
  $M$.

  Finally, we define $G^{00}_M$ to be the smallest type definable over $M$ subgroup of bounded index and  we define
  $G^{\infty}_M$ to be the smallest $M$-invariant subgroup of $G$ of bounded index.
\end{defi}

\begin{lemme}
Let $X$ be an f-generic definable set. Then $XX^{-1}$ is generic.
\end{lemme}
\begin{proof}
Let $(a_i:i<n)$ be a maximal sequence such that the sets $(a_i X:i<n)$ are disjoint, which must exist by f-genericity of $X$. Take any $b\in G$. Then for some $i<n$, $bX \cap a_i X \neq \emptyset$. Hence $b\in a_i XX^{-1}$ and $\bigcup_{i<n} a_i XX^{-1} = G$.
\end{proof}

\begin{lemme}\label{lem_bddindex}
Let $H<G$ be a type-definable group. Assume that $H$ is $\mu$-wide ({\it i.e.}, every definable set containing it is $\mu$-wide), then $H$ has bounded index.
\end{lemme}
\begin{proof}
Let $X$ be a definable set containing $H$. Then there is a definable set $Y$ containing $H$ such that $YY^{-1} \subseteq X$. By hypothesis, $Y$ is f-generic and the previous lemma implies that $YY^{-1}$ is generic and therefore $X$ is generic.
\end{proof}

In the following statement, $\mu_M$ is the ideal of formulas which do not extend to a global type, strongly f-generic over $M$. %Recall that $G^{00}_M$ is the smallest subgroup of $G$ which is type-definable over $M$ and of bounded index. Similarly, $G^{\infty}_M$ is the smallest subgroup of $G$ which is $M$-invariant and of bounded index.

\begin{thm}\label{th_stab}\
Assume that $G$ has strong f-generics. Let $p\in S_G(M)$ be f-generic.

Then $G^{00}_M=G^{\infty}_M=St_{\mu_M}(p)^2=(pp^{-1})^2$ and $G^{00}_M\setminus St_{\mu_M}(p)$ is contained in a union of non-wide $M$-definable sets.
\end{thm}
\begin{proof}
The ideal $\mu_M$ is $G$-invariant (by left multiplication), $M$-invariant and S1 on $G$ by Lemma \ref{lem_s1}. We can apply Theorem \ref{th_babystab} with hypothesis (B1) to deduce that $S=(pp^{-1})^2$ is a wide subgroup. As $p$ knows in which $G^{\infty}_M$ coset it lies, we must have $S\leq G^{\infty}_M$. On the other hand, by Lemma \ref{lem_bddindex}, $S$ has bounded index, hence $G^{00}_M\leq S$. It follows that those three subgroups are equal. The last statement also follows from Theorem \ref{th_babystab}.
\end{proof}

\subsection{Definably amenable groups}

A definable group $G$ is \emph{definably amenable} if for some (equiv. any) model $M$, there is a left-invariant Keisler measure on $M$-definable subsets of $G$. (See e.g. \cite[Chapter 8]{NIPbook}.)

\begin{fait}[\cite{NIPbook}, Lemma 7.5]\label{fact_meas}
Let $\mu$ be a measure over $M$ and $(b_i : i<\omega)$ an indiscernible sequence in $M$. Let $\phi(x;y)$ be a formula and $r>0$ such that $\mu(\phi(x;b_i))\geq r$ for all $i<\omega$. Then the partial type $\{ \phi(x;b_i) : i<\omega \}$ is consistent.
\end{fait}

\begin{prop}
Let $G$ be a definably amenable \ntp~group, then $G$ has strong f-generics.
\end{prop}
\begin{proof}
Fix a model $M$ and $\mu$ a $G$-invariant measure on $M$-definable
sets. Let $M\prec^{+} N$, and notice that it is enough to show
that $\mu$ extends to a measure over $N$ which is both
$G$-invariant and non-forking over $M$ (a type of positive
$\mu$-measure would be strong f-generic over $M$). So assume this
is not the case. By compactness, there are $\epsilon
>0$ and finitely many formulas $\phi_i(x;d)$, $i<n$, each forking
over $M$ such that any $G$-invariant extension $\tilde \mu$ of
$\mu$ satisfies $\bigvee_{i<n} \tilde \mu(\phi_i(x;d))>\epsilon$.
Take $(d_j:j<\omega)$ an indiscernible sequence in $\tp(d/M)$
which witnesses dividing as given by Fact \ref{fact_ntp}, (1). The
condition that $\tilde \mu$ extends $\mu $ and is $G$-invariant is
invariant under $Aut(N/M)$, therefore for every $j$, we also have
$\bigvee_{i<n} \tilde \mu(\phi_i(x;d_j))> \epsilon$. So up to
taking a subsequence, for some $i<n$, we have
$\bigwedge_{j<\omega} \tilde \mu(\phi_i(x;d_j))>\epsilon$. But
this contradicts Fact \ref{fact_meas} and the property of
$(d_j)_{j<\omega}$.
\end{proof}

\begin{cor}
Any solvable or pseudofinite \ntp group has strong f-generics.
\end{cor}

\section{PRC fields}\label{SPRC}

In this section we will give all the preliminaries in pseudo real closed fields that are required throughout the paper. The reader can see \cite{Pre}, \cite{Ba1}, \cite{Jarden} and \cite{Mon} for more details.
We give a useful description of definable sets which is more precise in the case of more variables that the description given in \cite{Mon}.

\begin{defi}
A field $M$ of characteristic zero \emph{is pseudo real closed
(PRC)} if $M$ is existentially closed (relative to the language of
rings) in every totally real regular extension $N$ of $M$.
Equivalently, if given any absolutely irreducible variety $V$
defined over $M$, if $V$ has a simple $\overline{M}^r-$rational
point for every real closure $\overline{M}^r$ of $M$, then $V$ has
an $M$-rational point.
\end{defi}

Prestel showed in Theorem 4.1 of \cite{Pre} that the class of PRC
fields is axiomatizable in the language of fields.

We have the following properties of PRC fields.

\begin{fact}\label{PRCcaracte}
Let $M$ be a PRC field.
\begin{enumerate}
    \item \emph{\cite[Proposition 1.4]{Pre}} If $<$ is an order on $M$, then $M$ is dense in $(\overline{M}^r, \overline{<}^r)$, the real closure of $M$ respect to the order $<$.
    \item \emph{\cite[Proposition 1.6]{Pre}} If $<_i$ and $<_j$ are different orders on $M$, then $<_i$ and $<_j$ induce different topologies.
    \end{enumerate}
\end{fact}

In this section we are interested in the class of bounded PRC fields. A field $M$
is \emph{bounded} if for any integer $n$, $M$ has finitely many
extensions of degree $n$. This implies in particular that all the
orders which make $M$ into an ordered field are definable (\cite[Lemma 3.5]{Mon}), and that there are finitely many of those.

\subsection{Preliminaries on bounded PRC fields}

We fix a bounded PRC field $K$  which is not algebraically closed
and a countable elementary substructure $K_0$ of $K$. So there is
$n \in \mathbb{N}$ such that $K$ has exactly $n$ distinct orders which are moreover definable (see Remark 3.2 of \cite{Mon}). Let
$\{<_1. \ldots, <_n\}$ be the orders on $K$.
If $n=0$, then $K$ is a PAC field, so we suppose from now on that $n\geq1$.

We will work over $K_0$, thus we denote by $\LCR$ the language of
rings with constant symbols for the elements of $K_0$, $\Li:= \LCR
\cup \{<_i\}$ and $\LC:= \LCR \cup \{<_1, \ldots, <_n\}$. We let
$T_{prc}:=Th_{\LCR}(K)=Th_{\mathcal L}(K)$. By Corollary 3.6 of \cite{Mon}, $T_{prc}$ is model
complete. If $M$ is a model of $T_{prc}$, we denote by
$\clos{M}{i}$ the real closure of $M$ with respect to $<_i$.

\bigskip

The following is a direct consequence  of the  ``Approximation
Theorem for $V$-topologies'' (\cite[Theorem 4.1]{PreZie}), and of
Fact $\ref{PRCcaracte}$.

\begin{fact}\label{SpTh}
 Let $(M, <_1, \ldots, <_n)$ be a model of  $T_{prc}$. Let $A$ be
 a subset of $M$ and for every order $<_i$ let $p^{(i)}$ be a quantifier-free
 $\Li$-type in $\clos{M}{i}$ (so a consistent set of polynomial $<_i$-inequalities).
Then $\bigcup_{i=1}^n p^{(i)}$ is a consistent type in $\LC$.
\end{fact}

Notice that the quantifier free $\LC$-types all have the same form
as the conclusion of Fact \ref{SpTh}.
We have the following amalgamation theorems for types:

\begin{fact}[\cite{Mon}, Theorem 3.21]\label{thamalgamation}
 Let $(M, <_1, \ldots, <_n)$ be a model of  $T_{prc}$. Let $E = \acl(E) \subseteq M$. Let $a_1, a_2, c_1,c_2$ be tuples of $M$ such that $E(a_1)^{alg}\cap E(a_2)^{alg}=E^{alg}$ and $\tp_{\LC}(c_1/E)=\tp_{\LC}(c_2/E)$. Assume that there is $c$ $ACF$-independent of $\{a_1,a_2\}$ over $E$ realizing $\qftp_{\LC}(c_1/E(a_1)) \cup \qftp_{\LC}(c_2/E(a_2))$.
Then $\tp_{\LC}(c_1/Ea_1) \cup \tp_{\LC}(c_2/Ea_2) \cup
\qftp_{\LC}(c/E(a_1,a_2))$ is consistent.
\end{fact}

\medskip

We now recall some other model theoretic properties of $T_{prc}$.

\begin{fact}[\cite{Mon}, Theorem 4.21]
The theory $T_{prc}$ is NTP$_2$.
\end{fact}

\begin{fact}[\cite{Mon}, Theorem 4.35]
In $T_{prc}$, all sets are extensions bases and forking equals dividing.
\end{fact}

\subsection{The multi-topology}

\begin{defi}
Let $(M, <_1, \ldots, <_n)$ be a model of $T_{prc}$, $A \subseteq
M$ and let $X \subseteq M^m$ be $\LCR(A)$-definable. Then
$\dim(X)= \max\{trdeg(\bar{x}/A): \bar{x} \in X\}$. This is a good
notion of dimension, since $\acl(A)= \dcl(A)= A^{alg} \cap M$
(\cite[Lemma 2.6]{Mon}). We will say that $\bar{a} \in X$ is a
\emph{generic point of $X$} over $A$ if $\dim(X) = trdeg(\bar{a}/A)$.
\end{defi}

\begin{defi}\textbf{(Multi-topology)}
 Let $(M, <_1, \ldots, <_n)$ be a model of $T_{prc}$.
Denote by $\tau_i$ the topology induced in $M$ by the order $<_i$.
By Fact \ref{PRCcaracte} (2), if $i \not = j$, then $\tau_i \not = \tau_j$.

A definable subset of $M$ of the form $I=
\displaystyle{\bigcap_{i=1}^n (I^i\cap M)}$ with $I^i$ a non-empty
$<_i$-open interval in $\clos{M}{i}$ is called a
\emph{multi-interval}.

Notice that by Fact \ref{SpTh} every multi-interval is non-empty
and if $I$ is a multi-interval, then $I$ is $<_i$-dense in each
$I^i$.

%Observe that by  \ref{ApTh} (Approximation Theorem) and Fact \ref{PRCcaracte} every multi-interval is non empty.
We define the \emph{multi-topology} $\tau$ as the topology in $M$
generated by the multi-intervals and  $\tau^m$ its product
topology in $M^m$. Observe that if $V$ is $\tau_i$-open, then it
is $\tau$-open. We call a \emph{multi-box in $M^m$} a set of the
form  $C= \displaystyle{\bigcap_{i=1}^n (C^i\cap M^m)}$, with
$C^i$ an $<_i$-box in $(\clos{M}{i})^m$.
\end{defi}

\medskip

We extend the definition of $(j_1, \ldots, j_r)$-cells for real
closed fields (see Definition 2.3 of \cite{Van2}) to find a
definition of multi-cells in the bounded PRC-field context.

\begin{defi}\textbf{(Multi-cells)}\label{Multi-cells}
Let $r \in \mathbb{N}$ and let $(j_1, \ldots, j_r)$ be a sequence
of zeros and ones of length $r$.

A \emph{$(j_1, \ldots, j_r)$-multi-cell} is definable subset $C$
of $M^r$ such that for every $i$ there is a $(j_1, \ldots,
j_r)$-cell $C^i$ in $\clos{M}{i}$ and
\[C= \displaystyle{\bigcap_{i=1}^n (C^i\cap
M^r)}.\]

\medskip

A \emph{multi-cell in $M^r$} is a $(j_1, \ldots, j_r)$-multi-cell,
for some $(j_1, \ldots, j_r)$.
\end{defi}

Observe that the $(1)$-multi-cells are multi-intervals and any
multi-box is a $(1,\ldots, 1)$-multi-cell.

Notice also that the \emph{open multi-cells in $M^r$} (or cells
which are open subsets of $M^r$) are precisely the
$\left(\underbrace{1, \ldots, 1}_{r}\right)$-multi-cells.

\begin{lem}\label{DimCells}
Let $m\in \mathbb{N}$ and let $(i_1, \ldots, i_m)$ and $(j_1,
\ldots, j_m)$ be two different sequences of zeros and ones of
length $m$. Let $C^i \in \clos{M}{i}$ be a $(i_1, \ldots,
i_m)$-cell and let  $C^j \in \clos{M}{j}$ be a $(j_1, \ldots,
j_m)$-cell. Then $\dim(C^i \cap C^j \cap M^m) < \min\{\dim(C^i),
\dim(C^j)\}$.
\end{lem}

\begin{proof}
Let $r_i= \dim(C^i)$ and $r_j = \dim(C^j)$.
Suppose that there is $\bar{a} = (a_1, \ldots, a_m) \in C $ such that $\bar{a}$ is a generic point of $C^i$ and $C^j$.
Let $X_i= \{a_k: i_k=0\}$, $X_j=\{a_k: j_k=0\}$. Then $r_i = m - |X_i|$ and $r_j = m - |X_j|$.
Observe that if $a_k \in X_i \cup X_j$, then $a_k \in \acl(a_1, \ldots, a_{k-1})$.
It follows that $\dim(C) \leq m - |X_i \cup X_j|$.
Since $(i_1, \ldots, i_m) \not = (j_1, \ldots, j_m)$, $|X_i \cup X_j| > \max\{|X_i|, |X_j|\}$.
Thus $\dim(C) \leq m - |X_i \cup X_j| < \min\{m - |X_i|, m - |X_j|\}= \min\{r_i, r_j\}$.
\end{proof}

It follows that for an intersection of two $r$-dimensional cells
to have dimension $r$, one needs that both cells have the same
sequences of $0$'s and $1$'s.

\begin{thm}\label{Cell-decom}
Let $(M, <_1, \ldots, <_n)$ be a  model of $T_{prc}$; let $A \subseteq M$, and $r \in \mathbb{N}$. Let $D \subseteq M^r$ be an $\LC(A)$-definable set in $M$.
Then there are $m \in \mathbb{N}$, and $C_1, \ldots, C_m$ with $C_j= \displaystyle{\bigcap_{i=1}^n (C^i_j\cap M^r)}$ a multi-cell in $M^r$ such that:
\begin{enumerate}
\item $D \subseteq \displaystyle{\bigcup_{j=1}^m C_ j}$;
\item $D \cap C_j$ is $\tau^r$-dense in $C_j$, for all $1 \leq j \leq m$;
\item for all $1 \leq i\leq n$ and $1\leq j\leq m$, $C^i_j$ is quantifier-free $\Li(A)$-definable in $\clos{M}{i}$;
\item for all $1 \leq i \leq n$ and $1 \leq j \leq m$, the set $C^i_j\cap M^r$ is $\Li(A)$-definable in $M$.
 \end{enumerate}
\end{thm}

\begin{proof}
The proof is by induction on the dimension of $D$. The case
$\dim(D)= 1$ follows from \cite[Theorem 3.13]{Mon}. Suppose that
$\dim(D) = d$. As in Theorem 3.13 \cite{Mon} using model
completeness of $T_{prc}$ we can suppose that there is an
absolutely irreducible variety $W$ defined over $\acl(A)$ such
that:
\[M \models \forall x_1, \ldots,x_r \left(\left(x_1, \ldots,x_r\right) \in D\right) \longleftrightarrow \left( \exists \bar{y}\
\left(x_1, \ldots,x_r, \bar{y}\right) \in
W^{sim}\left(M\right)\right),\] where $W^{sim}(M) = \{\bar{x} \in
W(M): \bar{x} \; \mbox{is a simple point of}\; W\}.$

Let $d = |\bar{y}|$, for each $i \in \{1, \ldots,n\}$ we define:
\[A_i := \{(x_1, \ldots, x_r) \in (\clos{M}{i})^r: \exists \bar{y} \in {(\clos{M}{i})}^d \text{ s.t. } (x_1, \ldots, x_{r}, \bar{y}) \in W^{sim}(\clos{M}{i})\}.\]

So $A_i$ is $\Li(A)$-definable and $D\subseteq A_i$.  By cell
decomposition in $\clos{M}{i}$, there are $k_i\in \mathbb{N}$,
$<_i$-cells $C^i_1, \ldots, C^i_{k_i}$ and $X^i$ such that:
\begin{enumerate}
 \item the sets $C^i_1, \ldots, C^i_{k_i}, X^i$ are quantifier free $\Li(A)$-definable in $\clos{M}{i}$;
\item $\dim(C_j^i)= d$, for all $j \in \{i_1, \ldots, i_{k_i}\}$;
\item $\dim(X^i)< d$, \item $A_i =
\displaystyle{\bigcup_{j=1}^{k_i}C_j^i} \cup X^i$.

 \end{enumerate}
 %For all $i \in \{1, \ldots, n\}$ and $j\in \{1, \ldots, r_i\}$, there is $(c^{1}_{i,j}, \ldots, c^{r}_{i,j})$ a sequence of zeros and ones such that $C^i_j$ is a $(c^{1}_{i,j}, \ldots, c^{r}_{i,j})$-cell.

Let $X= \displaystyle{\bigcup_{i=1}^n}(X^i \cap M^r)$ and let

\[J :=\{\sigma:\{1, \ldots, n\} \rightarrow \mathbb N \ |\
1\leq \sigma(i)\leq k_i\}.\]

\medskip

For all $\sigma \in J$, let $C_{\sigma}:=
\displaystyle{\bigcap_{i=1}^n{(C^i_{\sigma(i)}}}\cap M^r))$, so $D
\subseteq \displaystyle{\bigcup_{\sigma \in J}C_{\sigma} \cup X}$.
We are interested in $C_\sigma$ of maximal dimension $d$, so let
\[J' := \left\{\sigma \in J : \dim\left(C_{\sigma}\right)=d\right\}.\]

Let $\sigma\in J'$. By Lemma \ref{DimCells} all the cells $C^i_{\sigma(i)}$
must have the same sequences of $0$'s and $1$'s and therefore
$C_{\sigma}$ is a multi-cell in $M^r$.

\begin{claim}
 For all $\sigma \in J'$, $D\cap C_{\sigma}$ is $\tau^r$-dense in $C_\sigma$.
\end{claim}
\begin{claimproof}
Fix $\sigma \in J'$. Let $U_{\sigma}$ be a multi-box in $M^r$ such
that $V:= U_{\sigma} \cap C_{\sigma} \not = \emptyset$, we need to
show that $V \cap D  \not = \emptyset.$ Let $z \in V$. Then $z \in
A_i$ for all $i \in \{1, \ldots, n\}$. So there is $y^{(i)} \in
(\clos{M}{i})^{d}$, such that $(z, y^{(i)})$ is a simple point of
$W$. By Fact \ref{SpTh} we can find $(z_0, \bar{y_0}) \in W(M)$
such that $(z_0, \bar{y_0})$ is arbitrary $<_i$-close to $(z,
y^{(i)})$ for all $i\in \{1, \ldots, n\}$, in particular we can
find $z_0 \in V\cap D$.
\end{claimproof}

Let $Y= X \cup \displaystyle{\bigcup_{\sigma \in J \setminus
J'}C_{\sigma}}$, so $Y$ is an $\LC(A)$-definable set and
$\dim(Y)<d$.

Then $D \subseteq \displaystyle{\bigcup_{\sigma \in J'}C_{\sigma}
\cup Y}$ and each $C_{\sigma}$ satisfy $(2), (3)$ and $(4)$ of the
theorem. Since $\dim(Y)< d$, by induction hypothesis we can apply
the statement of the theorem to $Y$ instead of $D$, which
completes the proof.
\end{proof}

\begin{defi}
Let $(M, <_1, \ldots, <_n)$ be a model of $T_{prc}$ and $D\subseteq M^r$ a definable set.
Denote by $\overline{D}$ the closure of $D$ for the $\tau^r$-topology.
Observe that $\overline{D} = \displaystyle{\bigcap_{i=1}^n\overline{D}^{\tau_i}}$, where $\overline{D}^{\tau_i}$ is the closure of $D$ for the $\tau_i$-topology.
\end{defi}

If $X \subseteq M^r$ is a definable set and $C_1, \ldots, C_m$ are
the multi-cells obtained by Theorem \ref{Cell-decom}, then
$\displaystyle{\bigcup_{j=1}^m C_j} \subseteq \overline{D}$. This
implies the following corollary.

\begin{cor}\label{multiclosure}
Let $(M, <_1, \ldots, <_n)$ be a  model of $T_{prc}$, let $A
\subseteq M$, and $r \in \mathbb{N}$. Let $D \subseteq M^r$ be an
$\LC(A)$-definable set in $M$. Then there are $m \in \mathbb{N}$,
and $C_1, \ldots, C_m$ with $C_j= \displaystyle{\bigcap_{i=1}^n
(C^i_j\cap M^r)}$ a multi-cell in $M^r$ such that: $\overline{D} =
\displaystyle{\bigcup_{j=1}^m C_ j}$ and  such that for all $1
\leq i\leq n$, $C^i_j$ is quantifier-free $\Li(A)$-definable in
$\clos{M}{i}$, for all $1 \leq j \leq m$.
\end{cor}
\begin{proof}
The set $\overline{D}$ is $\LC(A)$-definable (so was $D$) and by Theorem
\ref{Cell-decom} there are $C_1, \ldots, C_m$ multi-cells in $M^r$
such that $\overline{D} \subseteq
\displaystyle{\bigcup_{j=1}^m}C_j \subseteq
\overline{(\overline{D})} = \overline{D}$. So $\overline{D}=
\displaystyle{\bigcup_{j=1}^m}C_j$.
\end{proof}

\begin{cor}\label{PRCAlgBound}
The theory $T_{prc}$ is algebraically bounded.
\end{cor}
\begin{proof}
Directly from Theorem \ref{Cell-decom}.
\end{proof}

\noindent \textbf{Notation.} Let $M$ be a structure and let $D
\subseteq M^r$ be a definable set. Let $k < r$. We define
\[\pi_k^M(D):= \{(x_1, \ldots, x_k) \in M^k: M \models \exists
x_{k+1}, \ldots, x_{r} \;(x_{1}, \ldots, x_{r}) \in D \}.\]

For $\bar{\alpha} \in \pi_k^M(D)$, define
$D^M_{\bar{\alpha}}:=\{\bar{y} \in M^{r-k}: (\bar{\alpha}, \bar{y}) \in D\}.$
Define $D^M(a_1, \ldots, a_k):= \{(a_1, \ldots, a_k, x_{k+1},
\ldots, x_r): M \models (a_1, \ldots, a_k, x_{k+1}, \ldots, x_r)
\in D\}$.
We omit $M$ when the structure is clear.

\section{Type definable subgroups of algebraic groups}\label{SAlgebraic}

We wish to apply Proposition \ref{TdefGps} to analyze type
definable groups of algebraic groups in bounded PRC fields. For
this, we need to develop the theory of externally definable sets
in bounded PRC fields.

We will show that expanding a bounded PRC field with certain
\emph{externally definable} sets has elimination of quantifiers,
analogous to results in \cite{BaPo} and \cite{Sh783}.

\begin{defi}
Let $T$ be a theory and let $M$ be a model of a theory $T$. An \emph{externally definable subset of $M^k$}
is an $X \subseteq M^k$ that is equal to $\varphi(N^k, d)\cap M^k$ for some formula $\varphi$ and $d$ in some $N \succeq M$.

We denote by $M^{Sh}$ (\emph{the Shelah expansion of $M$}) the structure obtained from $M$ by naming all the externally definable sets.
\end{defi}

\begin{defi}

Let $(M, <_1, \ldots, <_n)$ be a model of $T_{prc}$.
We say that $C = \bigcap_{i=1}^n C^i \cap M^r$ is an \emph{externally definable multi-cell in $M^r$} if for $i\in \{1,\ldots, n\}$, $C^i$ is the trace on $(\clos{M}{i})^r$ of a cell defined with exterior parameters.
We say that $C$ is a \emph{multi-cell externally $\mathcal{L}(N)$-definable} if $N\succeq M$ and for each $i \in \{1, \ldots, n\}$, there is $\phi_i(\bar{x}) \in \mathcal{L}(\clos{N}{i})$ such that $C^i = \phi_i(M)$.

\end{defi}

Baisalov and Poizat prove in \cite{BaPo} that the theory
resulting in expanding the language of any o-minimal structure
with externally definable sets has elimination of
quantifiers. This was generalized by Shelah to all NIP theories in
\cite{Sh783}.

%\begin{lemma}
%Let $R$ be a model of RCF, let $U\subseteq R^n$ be an externally definable open set and let $f: U \to R$ be an externally definable function (that is, the graph of $f$ is externally definable). Then there is an externally definable partition $U=U_1 \cup \cdots \cup U_m$ of $U$ and definable functions $g_k: V_k \to R$, $k\leq m$, with $U_k \subseteq V_k$ and such that $f|_{U_k} = g|_{U_k}$.
%\end{lemma}
%\begin{proof}
%
%\end{proof}

\begin{prop}\label{prop_extrcf}
Let $R$ be a model of RCF in the ring language $\mathcal L_{ring}$.% and consider the expansion $R^{Sh}$ of $R$ obtained by naming all externally definable sets.
Then any definable subset of $R^{Sh}$ can be written as a finite union of sets of the form $U\cap D$, where $U$ is an open externally definable subset and $D$ is $\mathcal L_{ring}$-definable.
\end{prop}
\begin{proof}
By \cite{BaPo} the structure $R^{Sh}$ is weakly o-minimal, so it makes sense to consider dimensions of definable sets. Let $X\subseteq R^n$ be definable in $R^{Sh}$. We prove the result by induction on the dimension of $X$. If $X$ has dimension 0, then it is finite, and the result follows.

For the inductive case, we can write $X$ as the union of an open set and a set of lower dimension, so we can assume that $X$ has dimension $d<n$. Let $\pi:R^n \to R^d$ be a coordinate projection such that $\pi(X)$ has non-empty interior (see Theorem 4.11 of \cite{MMS}). Then again writing $\pi(X)$ as the union of an open set and a set of smaller dimension, we may assume that $\pi(X)$ is open. For each $\bar a \in \pi(X)$, the fiber $X_{\bar a}$ is finite. By decomposing $X$ further, we may assume that it has always exactly one element. So $X$ is the graph of a function from $U:=\pi(X)$ to $R^{n-d}$.

Let $R^{Sh} \prec R'$ be a sufficiently saturated elementary
extension. Then by Proposition 1.7 in \cite{ChSi2}, there is an
$\mathcal L_{ring}(R')$-definable set $X' \subseteq R'$ such that
$X'(R') \subseteq X(R')$ and $X'(R) = X(R)$. Hence $X'$ is also
the graph of a function from some $R'$-definable set $V$ to
$R^{n-d}$, with $V(R)=U(R)$. As we are working in RCF, up to
decomposing $V$ in finitely many $R'$-definable sets, we may
assume that $f'$ is the function sending a point $\bar a\in V$ to
the $k$-th solution of $P(\bar b,\bar a,\bar Y)$, where $P(\bar
b,\bar T,\bar Y)$ is a polynomial with coordinates $\bar b\in R'$.
Since by hypothesis, $P(\bar b,\bar T,\bar Y)$ has a solution in
$R$ for each $\bar a$ in the open set $U$, $P$ is definable over
$R$. This implies that $X$ coincides on $U$ with the graph
$\Gamma$ of an $R$-definable function. Then $X=U\times R^{n-d}
\cap \Gamma$ has the required form.
\end{proof}

We now aim to show that the expansion of a bounded PRC field
in $\LC$ by externally definable multi-cells has elimination of
quantifiers and is NTP$_2$.

\begin{prop}\label{UniformDecom}
Let $(M, <_1, \ldots, <_n)$ be a model of $T_{prc}$. Let $A
\subseteq M$ and let $D\subseteq M^r$ be
$\mathcal{L}(A)$-definable. Then there are $m \in \mathbb{N}$ and
$C_1, \ldots, C_m$ multi-cells in $M^r$, $\mathcal{L}(A)$-definable such
that $D \subseteq \bigcup_{j=1}^m {C_j}$ and such that for every
$\bar{x} \in \pi_{r-1}(D\cap C_j)$ the fiber $D_{\bar x}$ is $\tau$-dense
in $(C_j)_{\bar x}$.
\end{prop}
\begin{proof}

Notice that if $D=D_1 \cup D_2$ and the theorem is known for $D_1$ and $D_2$, then it follows for $D$ by taking a common refinement of the two cell decompositions obtained for $D_1$ and $D_2$.

Let $D$ be a definable set. By Theorem \ref{Cell-decom} for any  $\bar{x} \in \pi_{r-1}(D)$
there are $k_{\bar{x}}$, $U_1, \ldots, U_{k_{\bar{x}}}$
multi-intervals in $M$ and a finite set $B_{\bar{x}}$ such that
$D_{\bar{x}} \subseteq \displaystyle{\bigcup_{j=1}^{k_{\bar{x}}}
U_{\bar{x},j}} \cup B_{\bar{x}}$, and such that $D_{\bar{x}}$ is
$\tau$-dense in $U_{\bar{x},j}$, for all $j \in \{1, \ldots,
k_{\bar{x}}\}$. By definition of multi-intervals $U_{\bar{x},j}=
\displaystyle{\bigcap_{i=1}^n}U_{\bar{x}, j}^i \cap M$, where
$U_{\bar{x}, j}^i$ is a $<_i$-interval in $\clos{M}{i}$.

\medskip

For all $m_1, m_2\in \mathbb N$, let $A_{m_1,m_2}:= \{\bar{x} \in
\pi_{r-1}(D): k_{\bar{x}} = m_1 \; \mbox{and} \; |B_{\bar{x}}|=m_2
\}$. By compactness there are only finitely many $(m_1,m_2)$ for
which $A_{m_1,m_2}$ is non empty.

Then $A_{m_1, m_2}$ is definable with the same parameters as
$D$, and $\pi_{r-1}(D)= \bigcup_{(m_1, m_2)} A_{m_1, m_2}$ (a
finite union). Since $D= \bigcup_{(m_1, m_2)} \pi_{r-1}^{-1}(
A_{m_1, m_2})$, it is enough to show that each
$\pi_{r-1}^{-1}( A_{m_1, m_2})$ can be decomposed according to the
conclusion of the theorem, so assume that $D=\pi_{r-1}^{-1}
( A_{m_1, m_2})$ for some $(m_1, m_2)$.

\medskip

Let $i \in \{1, \ldots, n\}$. For $s \in \{1,2,3\}$, let $f_{s,j}^i(x): A_{m_1, m_2}\mapsto \clos{M}{i}$ such that:

\begin{enumerate}
 \item $f^i_{1,j}(\bar{x}) = y$ if and only if $y$ is the ``$<_i$-smallest extremity in $\clos{M}{i}$'' of the $<_i$-interval $U^i_{\bar{x}, j}$.
 \item $f^i_{2,j}(\bar{x}) = y$ if and only if $y$ is the ``$<_i$-largest extremity in $\clos{M}{i}$'' of the $<_i$-interval $U^i_{\bar{x},j}$.
\item $f^i_{3,j}(\bar{x}) = y$ if and only if $y$ is the $j$-th
point in $B_{\bar x}$ in the order $<_i$.
\end{enumerate}

As $T_{prc}$ is algebraically bounded (see Corollary \ref{PRCAlgBound}), there is a definable partition of the base $A_{m_1,m_2} = \bigcup_{t<p} X_t$ such that on each $X_t$, each of the functions $f^ i_{s,j}$ coincides with a $<_i$-semi-algebraic function. Decreasing $D$ further, we may assume that $p=1$ and that all the functions $f^ i_{s,j}$ are semi-algebraic.

Now, let
\[C_{j}:=\{ (\bar{x}, y) \in M^r:  f^i_{1,j}(\bar{x}) < y < f_{2,j}^i(\bar{x}), \text{ for all }i \}\] and
\[C_{j}^0:=\{ (\bar{x}, y) \in M^r: f^1_{3,j}(\bar{x}) = y \}.\]

%Recall that we assumed $D$ was the restricted domain $A_{(m_1,
%m_2)}$ of $\bar x$'s such that $D_{\bar x}$ was $<_i$-dense in
%$m_1$ many intervals and had $m_2$ many $<_i$-isolated points.

Then $D\subseteq \bigcup_j C_{j} \cup \bigcup_j C_{j}^0$ and this decomposition has the required properties.
\end{proof}
%
%\begin{thm}\label{ExtSet}
%Let $(M, <_1, \ldots, <_n)$ be a model of $T_{prc}$. Let $N \succeq M$ such that $N$ is $|M|^+$-saturated.
%Let  $U\subseteq M^r$ be a multi-cell externally $\mathcal{L}(N)$-definable, and let $D\subseteq M^r$ be a $\mathcal{L}(M)$-definable set.
%Then there are a definable set $D^*$ and a  multi-cell externally $\mathcal{L}(N)$-definable $U^*$, such that $\pi_{r-1}(U \cap D) = U^* \cap D^*$.
%\end{thm}
%\begin{proof}
%
%\end{proof}

%\begin{lemme}\label{lem_honestdef}
%Let $M_0\models T_{prc}$ and let $M_0\preceq N_0$ and $\phi(\bar x;\bar d)$ define a multi-cell. Let $(N_0,M_0)$ denote the expansion of $N_0$ with a new unary predicate naming $M_0$ and take $(N_1,M_1)\succ (N_0,M_0)$ an $|N_0|^+$-saturated elementary extension. Then there is $\bar e\in M_1$ and a formula $\psi(\bar x;\bar e)$ defining a multi-cell such that $\psi(M_1;\bar e)\subseteq \phi(M_1;\bar d)$ and $\psi(M_0;\bar e)=\phi(M_0;\bar d)$.
%\end{lemme}
%\begin{proof}
%First note, that if $M\models T_{prc}$ is $\kappa$-saturated, then also $M^{(i)}$ is $\kappa$-saturated: all cuts of cofinality $<\kappa$ are filled since $M$ is dense in $M^{(i)}$ and this implies saturation by o-minimality.
%
%Rewriting $\phi(\bar x;\bar y)$, we can assume that it is a positive boolean combination of quantifier-free formulas each in one of the languages $\Li$.
%\end{proof}

\begin{defi} \label{ShelahExpPRC}
Let $\mathcal{U}$ be a monster model of $T_{prc}$. Let $M$ be a
model of $T_{prc}$. Let $N \succeq M$ such that $N$ is
$|M|^+$-saturated. Then $\clos{N}{i}\succeq N$, and $\clos{N}{i}$
is $|\clos{M}{i}|^+$-saturated, for all $i \in \{1, \ldots, n\}$.

Let $\mathcal{L}^* = \mathcal{L} \cup  \{R_C(\bar{x}) : C \; \mbox{is a multi-cell externally} \; \mathcal{L}(N)\mbox{-definable}\} \cup \{P_D(\bar{x}) : D \; \mbox{is} \; \mathcal{L}(M)\mbox{-definable}\}$.
We define $M_N$ to be the structure in the language $\mathcal{L}^*$
whose universe is $M$ and where each $R_C$ and each $P_D$ are interpreted as:
\begin{enumerate}
\item for every $\bar{a} \in M, M_N \models R_C(\bar{a})$ if and only if $\mathcal{U} \models \bar{a} \in C$,
\item for every $\bar{a} \in M, M_N \models P_D(\bar{a})$ if and only if $M \models \bar{a} \in D$.
\end{enumerate}

\end{defi}

\begin{rem}
Observe that $M_N$ is not $M^{Sh}$ because we only add predicats for the externally definable multi-cells, not for all the externally definable sets.

\end{rem}

\begin{thm}\label{QEShelahExpPRC}
The structure $M_N$ admits elimination of quantifiers.
\end{thm}
\begin{proof}
Let $C$ be an externally $\mathcal L(N)$-definable multi-cell and $D$ an $\mathcal L(M)$-definable set, both inside some $M^r$. Let $\pi$ be the projection to the first $r-1$ coordinates. It is enough to show that $\pi(C\cap D)$ is quantifier-free definable in $M_N$.

First, write $C = \bigcap_{i=1}^n C^i \cap M^r$, where each $C^ i$ is an externally definable multi-cell in $(\clos{M}{i})^r$. By Proposition \ref{prop_extrcf}, we can write each $C^ i$ as a finite union of sets of the form $U^ i \cap D^ i$, where $U^ i$ is an externally definable open subset of $(\clos{M}{i})^r$ and $D^ i$ is definable in $\clos{M}{i}$. Then the trace of $U^ i$ on $M^r$ is also open by density of $M^r$ in $(\clos{M}{i})^r$ and the trace of $D^ i$ on $M^r$ is definable in $M$. The result we want to prove is stable under taking finite unions, so we may assume that $C^ i = U^ i\cap D^ i$ and then by integrating $D^ i$ into $D$, we may assume that $C=C^ i \cap M^r$ is open in $M^r$.

By Proposition \ref{UniformDecom}, we may assume that $D$ is $\tau$-dense in some multi-cell $C_*$ which contains it and such that if $\bar x \in \pi(D)$, then $D_{\bar x}$ is $\tau$-dense in the fiber $(C_*)_{\bar x}$. As $C$ is $\tau$-open, for any $\bar x\in \pi(D)$, if the fiber $(C\cap C_*)_{\bar x}$ is non-empty, then it is open in $C_{\bar x}$ and thus also $(D \cap C)_{\bar x}$ is non-empty. Therefore $\pi(D\cap C)=  \pi(D)\cap \pi(C\cap C_*)$ is quantifier-free definable in $M_N$.
\end{proof}

\begin{cor}\label{ShelahExpNTP}
The structure $M_N$ is NTP$_2$.
\end{cor}
\begin{proof}
This follows from Theorem \ref{th_shelahexp} proved in the appendix.
\end{proof}

%\begin{lemme}
%Let $T$ be NIP in a language $L$, $G$ a definably amenable group for which $G^{00}$ is externally definable. Let $M\models T$ be $|T|^+$-saturated. Then in the Shelah expansion $M^{Sh}$, there is a type $p(x)$ over $M^{Sh}$ which concentrates on the definable set $G^{00}$ and is invariant by translation by elements of $G^{00}$.
%\end{lemme}
%\begin{proof}
%Let $D(x)$ denote $G^{00}(M)$ as a definable set in $M^{Sh}$. Assume the conclusion does not hold, then by compactness, there are $M^{Sh}$-definable sets $X_i$, $i<n$ and elements $g_i \in G^{00}$ such that $D(x)\wedge \bigwedge X_i(x)\leftrightarrow X_i(g_i x)$ is inconsistent. Let $M_0\prec M$ have size $|T|$ and take $p_0(x)$ a global type strongly f-generic over $M_0$.

%Let $M^{Sh} \prec N$ be an $|M|^{+}$ saturated expansion and denote by $X(x)$ the $M^{Sh}$-definable set corresponding to $G^{00}$. Take also $M_0\prec M$ a submodel of size $|T|$. Let $p_0(x)$ be a global $M_0$-invariant $L$-type which concentrates on $G^{00}$ and is $G^{00}$-invariant. Let $\phi(x)\in p_0|N$, then we can find $b\in X(N)$ such that $b\models \phi(x)$: indeed, by invariance, we may assume that
%
%Given a set $B\subseteq N$ of size $\leq |T|$, we can find $b\in X(N')$ such that $b\models p_0 | B$: Let $\phi(x)\in p_0|B$, then By NIP, we can find a maximal Morley sequence $(a_i:i<\kappa)$ of $p_0$ inside $X(N)$.
%\end{proof}

\begin{lemme}\label{RCFInvType}
Let $(G, \star)$ be an algebraic group in an $\aleph_1$-saturated real closed field $(R,<)$.
Then there is an externally definable $<$-open subgroup $H\leq
G$ which has an invariant definable type in the expansion
$R^{Sh}$, where we expand the language to include all the
$R^*$-definable subsets of $R$ for some saturated $R^*\succeq R$.
\end{lemme}

\begin{proof}
As in \cite[Proposition 7.8]{NIP1}, we identify a small
neighborhood of $e$ in $G$ with a neighborhood of zero in $R^n$. If we let $\epsilon$ be
infinitesimal with respect to $R$, then we have $$|x \star y -
(x+y)|\leq C|(x,y)|^2$$ for some $C\in R$ and all $|x|,|y|\leq \epsilon$.
Let $U$ be the convex set of infinitesimals with respect to $R$.
Then $H = \{\bar x, x_i\in U$ for all $i\}$ is a subgroup of $G$.

The set $U$ is definable using parameters in $R^*$, so $U$ is defined by a predicate
$\tilde U$ in $R^{Sh}$ and therefore  $H$ is also defined by a predicate $\tilde H$.

Let $\tilde V$ denote the
set of elements $x$ in $R$ such that $x\geq 1/n$ for some $0<n<\omega$, which by compactness and saturation is also the trace in $R$
of an $R^*$-definable set, so it is definable in $R^{Sh}$.

Let $p(x_1,\ldots,x_n)$ be the type in $\tilde H$ saying that
$x_1$ is as large as possible in $\tilde U$, and for all $k>1$,
$x_k/x_{k-1}$ is infinitely small in $\tilde V$. Using weak o-minimality
of $R^{Sh}$ we know that $p$
determines a (definable) complete type.

We will show that $p$ is $\tilde H$-invariant, so that $\tilde H(R)$ and $p$ satisfy the statement of the lemma. Let $\bar a
=(a_1,\ldots, a_n)\in \tilde H(R^*)$ and let $\bar b$ realize $p$
over $R^*$. We have to show that $\bar y := \bar a \star \bar b$
realizes $p$ over $R^*$.

All coordinates of $\bar a$ and
all $b_k^2$, are infinitesimal with respect to each $b_k$, so
$\bar a \star \bar b = \bar a + \bar b + \bar \epsilon$, where
$|\bar \epsilon| \leq C\cdot b_1^2$. Now $y_1 = b_1 + a_1
+\epsilon_1$, $a_1\in R^*$, $b_1$ as as large as possible in $U$
and $|\epsilon_1|\leq b_1^2$ which is much less than $b_1$, so
$\tp(y_1/R^*)\in U$ satisfies $\tp(b_1/R^{*})$.

In the same way, we have
$$\frac{y_k}{y_{k-1}} = \frac
{b_k+a_k+\epsilon_k}{b_{k-1}+a_{k-1}+\epsilon_{k-1}}$$ hence
$$\frac{1/2 b_k}{2b_{k-1}} \leq \frac {y_k}{y_{k-1}} \leq
\frac{2b_k}{1/2 b_{k-1}}$$ from which it follows that
$y_k/y_{k-1}$ realizes over $R$ the type of an infinitesimally
small element in $\tilde V$. So $\bar y$ realizes $p$, as required.
\end{proof}

\begin{prop} \label{PRCProfinite}
Let $M$ be a model of $\PRCB_{prc}$. Let $G$ be an algebraic group
definable in $M$, let $K\leq G$ be a type definable subgroup and
$L= \overline{K}$. Then $K$ has bounded index in $L$, and $L/K$
with the logic topology is profinite.
\end{prop}

\begin{proof}

Let $\overline{K}^z$ be the Zariski closure of $K$. Then
$\overline{K}^z$ is an algebraic subgroup of $G$,
$\overline{K}^z$ is type-definable and $\dim(\overline{K}^z)=
\dim(K)$.

So replacing $G$ by $\overline{K}^z$ we can suppose that $\dim(G)
= \dim(K) := m$. Observe that $K$ has bounded index in $L$.

Let $N \succ M$ be $|M|^+-$saturated. We now work in the structure
$M_N$ defined in Definition \ref{ShelahExpPRC}. It is NTP$_2$ by
Corollary \ref{ShelahExpNTP}. Suppose we have $n$-definable
orders. For each $i$, we will define, in the ordered $<_i$-ring
language $\clos{\mathcal L}{i}$ (using externally definable sets)
a definable set $H^i$, and a type $p^i$ in $\clos{N}{i}$ as
follows.

For each $i \in \{1, \dots, n\}$, let $H^i$ be the $\mathcal
L^{ext}$-definable subgroup  of $G$, and let $p^i$ be the
invariant $\mathcal L^{ext}$-definable type given by Lemma
\ref{RCFInvType}. So $H^i$ is the trace of an
$\clos{N}{i}$-definable set in $\clos{M}{i}$.

Let $H:= \displaystyle{\bigcap_{i=1}^n} H^ i \cap M$, and let $p
:= \displaystyle{\bigcup_{i = 1}^n p^i}$. By Fact \ref{SpTh}, $H
\not = \emptyset$ and $p$ is finitely consistent in $M$.

We have that $H$ is an externally definable set in $M$, and each
$p^i$ is definable in $(\clos{N}{i})^{ext}$, so $p$ is a definable
partial type in $M_N$. In a similar way we also obtain that $p$ is
$H$-invariant.

So $H$ is a $\tau$-open definable subgroup of $L$, and since $K$
is $\tau$-dense in $L$, all the cosets intersect $K$ and we obtain
that $H/H\cap K\cong L/K$.

We define an ideal $\mu$ over $H$ by $X \in \mu$ if $\overline{X}
\not \in p$. This ideal is definable and $H$-invariant.
%We can take $p$ of maximal dimension with all the above properties. \textbf{Pierre: I do not understand what this sentence means.}

 \begin{claim}
$\mu$ is S1 over $H$.
\end{claim}

\begin{claimproof}
If $X$ is a definable set, by Theorem \ref{Cell-decom} and
Corollary \ref{multiclosure} it follows that $\overline{X}\in p$
if and only if $X \cup p$ is consistent. Let $\phi(x,y)$  be a
formula and let $(a_j)_{j \in \omega}$ be indiscernible over $H$
such that $\phi(x, a_j) \not \in \mu$, for all $j\in \omega$. Then
all of the formulas $\overline{\phi(x, a_j)}$ are in $p$, and for
each $j$ we have that $\phi(x,a_j) \cup p$ is consistent.

Let $c_1$ and $c_2$ be such that $c_1 \models \phi(x, a_1) \cup p$, and $c_2 \models \phi(x, a_2) \cup p$, and such that $c_1$ and $c_2$ are algebraically independent over $\{a_1, a_2\}$.
By Fact \ref{thamalgamation} $ \tp(c_1/a_1) \cup \tp(c_2/a_2) \cup p$ is consistent.
It follows that $\phi(x, a_1) \cup \phi(x, a_2) \cup p$ is consistent, and by $\tau$-completeness of $p$ we have
\[\overline{\phi(x, a_1) \wedge \phi(x, a_2)} \in p.\]

By indiscernibility \[\overline{\phi(x, a_i) \wedge \phi(x, a_j)} \in p\] for all $i \not = j$, so $\phi(x, a_i) \wedge \phi(x, a_j) \not \in \mu$, for all $i \not = j$.
\end{claimproof}

Now, ${\overline{H \cap K}}\in p$ so that $H \cap K$ is
$\mu$-wide. It follows by Theorem \ref{TdefGps} that $H \cap K$ is
an intersection of definable groups. Hence $H/H\cap K$ with the
$\mathcal L^*$-logic topology (see Definition \ref{ShelahExpPRC})
is profinite, and then so is $L/K$ which is isomorphic to it.

The $\mathcal L$-logic topology on $L/K$ is compact and Hausdorff
and is weaker than the $\mathcal L^*$-logic topology which is also
compact and Hausdorff. It follows that both topologies coincide.
In particular $L/K$ with the $\mathcal L$-logic topology is
profinite so that $K=\bigcap G_i$ where $G_i$ is $\mathcal
L$-definable and $G_i\cap L$ is a subgroup of $L$.
\end{proof}

%Let $\widetilde{H}= H \cap \displaystyle{\bigcap_{i =1}^n \{ (-\epsilon_i, \epsilon_i)_i^m}: \epsilon_i \in \mathcal{U}\}$.

%Then $\widetilde{H}\leq L$ and is type-definable (in $M^*$).
%Observe that $\widetilde{H}/\widetilde{H}\cap K\cong L/ K$(Manque)

%We only need to show that $\widetilde{H}/\widetilde{H}\cap K$ is profinite.

\section{Definable groups with f-generics in PRC}\label{STheorem}

\begin{defi}
Let $(M, <_1, \ldots, <_n)$ be a model of $T_{prc}$. We say that a
definable set $X \subseteq M^m$ is \emph{multi-semialgebraic} if
$X$ is a union of multi-cells in $M^m$. Let $(G, \cdot_G)$ be an
$M$-definable group. We say that $G$ is \emph{multi-semialgebraic}
if $G$, the graph of $\cdot_G$ and of the inversion of $G$ are
multi-semialgebraic.
\end{defi}

%This section is devoted to proving the following theorem.

\begin{thm}\label{th_main}
Let $M \models T_{prc}$ be $\omega$-saturated. Let $G$ be an
$M$-definable group with strong f-generics. Then there is a finite
index $M$-definable subgroup $G_1 \leq G$, a finite $K\leq
G_1$ central in $G_1$, and an algebraic group $H$ such that there
is a local group homomorphism from a generic subset $W_1^*$ of
$G_1$ and a ``finite index'' subset of a $\tau$-open neighborhood
$W_1$ of the identity of $H(M)$.

Furthermore, $G_1(M)$ is definably isomorphic to a
finite index subgroup of a multi-semialgebraic group $H'(M)$,
where $H'(M)$ admits a definable $\tau$-manifold structure, where
each open set in the cover maps via a definable local group
homeomorphism to a neighborhood $U_3$ of the identity of $H(M)$.
\end{thm}

\begin{proof}
Let  $\mu_M$ be the ideal of formulas which do not extend to a
strongly bi-f-generic type over $M$ (which exist by Lemma
\ref{lem_biexist}). So $\mu_M$ is $M$-invariant, S1 (by Lemma
\ref{lem_s1}),  and invariant under both left and right
translations by elements of $G$. Let $q\in S(M)$ be $\mu_M$-wide.

By Theorem \ref{th_stab}, $Stab(q) = G^{00}_M$ and $\mu_M$-almost all elements of $G^{00}_M$ are in $St(q)$.

Let $a\in G^{00}_M$ be such that $\tp(a/M)$ is $\mu_M$-wide. Let $b\models q$ such that $\tp(b/Ma)$ is $\mu_M$-wide and $\tp(ab/M)=q$.

By Theorem \ref{algebraic group chunk}, there is an algebraic
group $H$, and a type definable subgroup $K$ of $G\times H$ such
that $\pi_1(K)$ contains $G^{00}_M$ and $\pi_2(D)$ is of finite
index.

As in the proof of \ref{algebraic group chunk}, we can assume that
$\pi_2$ is injective on $K$ (otherwise we can replace $G$ by
$G/\pi_2(K_2)$ where $K_2 = \pi_2^{-1}(e)\cap K$).

Now choose a symmetric definable $X_0$ such that $K \subseteq X_0
\subseteq G\times H$, and  such that $\pi_1$ and $\pi_2$ are
injective on $X_0^4$. Replacing $G$ by a subgroup of finite index,
we can assume that $\pi_1(X_0)$ generates $G$. By Proposition \ref{prop_normal}, $K$ is normal in the group
generated by $X_0$ and $X_0^n$ is medium for any $n$. Observe
that $\pi_1(X_0)\subseteq G$ is definably isomorphic to
$\pi_2(X_0)\subseteq H$.
In $\pi_2(X_0)$, the multi-topology $\tau$ is definable and the operations in $H$ are continuous.

Since working with projections becomes quite messy, we will use
notation in the following way:

\begin{itemize}
\item Any element of $X_0$ will be written as $(x^*, x)$ where
$x^*\in G$ and $x\in H$. So $x^*=\pi_1(\pi_2^{-1}(x))$ for any
element $x$ in $\pi_2(X_0)\subseteq H$. We will also do this for
sets, so that $A^*=\pi_1(\pi_2^{-1}(A))$ for any $A\subset
\pi_2(X_0)$.

\item All non $^*$-elements will be assumed to belong to $H$. We
will use Greek letters for elements in $G$ which may not be in
$\pi_1(X_0)$.

\item We will refer to $\pi_2(K)$ by $K_H$.

\item We will mostly be working inside $H$, so we will drop the
index in $\cdot_H$.

%\item From now on, when it does not present confusion, we will also identify $X$ with its copy in $H$, so that $X=(X^*,X)$.
\end{itemize}

\bigskip

In $H$ we have that $\pi_2(X_0) \cap \overline{K_H}$ is generic in
$\overline{K_H}$ and by Proposition \ref{PRCProfinite},
$\overline{K_H}/K_H$ is profinite, so there is a $\LC$-definable
set $X \subseteq \pi_2(X_0)$  such that $K_H':= X \cap
\overline{K_H}$ is a subgroup of finite index  of
$\overline{K_H}$, and $(K_H')^*$ is a type-definable subgroup of
bounded index of $G$ ($G^{00}_M \subseteq (K_H')^*$). By passing
to a finite index subgroup of $G$, we may assume that $X^*$
generates $G$.

\begin{claim}
We may assume that $K_H'$ is a normal subgroup of $\overline{K_H}$, in fact normalized by $X$, and that $(K_H')^*$ is a normal subgroup of $G$.
\end{claim}

\begin{claimproof}
Let $r$ be the smallest integer such that every $\gamma^* \in G$ is the $\cdot_G$-product of $r$ elements in $X^*$.
Define $Y_1, \ldots, Y_r$ (in $H$) such that:

\[Y_1:= \bigcap_{\gamma \in X} \gamma X {\gamma^{-1}},\]
\[Y_l:= \bigcap_{\gamma \in X} \gamma Y_{l-1} \gamma^{-1}, \mbox{for}  \; 2 \leq l \leq r.\]

Then $(Y_r)^*$ is normalized by $G$, hence $Y_r$ is normalized by $X$ and so is $Y_r \cap \overline{K_H}$. We can now replace $X$ by $Y_r$ and $K_H'$ by $Y_r \cap \overline{K_H}$.
\end{claimproof}

Now, $\overline{K_H}$ is the intersection of multi-semialgebraic
sets in $H$. We can define a decreasing sequence $(U_k:k<\omega)$
of quantifier-free definable symmetric sets, such that:

-- $U_0=\overline{X}$;

-- $({U_{m+1}\cap X})^3\subseteq {U_{m}\cap X}$, for each $m<\omega$;

-- $g({U_{m+1}\cap X})g^{-1} \subseteq U_m
\cap X$, for each $m<\omega$ and $g\in X$;

-- $\overline{K_H}=\bigcap_{k \in \omega} U_k$.

Note that by density of $X$ in $U_0$ and by continuity of the operations, we also have $({U_{m+1}})^3
\subseteq {U_{m}}$ and $g{U_{m+1}}g^{-1} \subseteq U_m$ for all $g\in X$.

\begin{claim}
We may assume that $U_m$ are multi-open, for $m\geq 1$.
\end{claim}

\begin{claimproof}
The type definable group $\overline{K_H}$ has non empty interior in $H$ (since it has bounded index). The operations are continuous in $H$ and by definition $X$
is dense in $U_0$. It follows that, since $U_{m+1}\cdot
\OK_H\subseteq U_{m}$, every point in $U_{m+1}$ has a neighborhood
contained in $U_m$, so $U_{m+1}$ is entirely contained in the
interior of $U_m$. Replacing each $U_m$ by its interior, we preserve the properties and the claim holds.
\end{claimproof}

We have a local group homomorphism between $(U_3\cap X)^*$ and
$U_3\cap X$.

The rest of the proof will be devoted to define a group $W/E$
isomorphic to $G$ which can be covered by finitely many copies
$W_i$ of $U_3\cap X$, then look at the ``open closure'' $W_i^{cl}$
of each $W_i$ which will be isomorphic to $U_3$, and finally, we
will induce a group structure on the ``corresponding'' group
$ W^{cl}/E^{cl}$ which will then give us the group $H'$
as in the statement of the theorem.

Select points $\{\alpha_k :k<p\}$ in $G$ such that \[G =
\displaystyle{\bigcup_{k<p} \alpha_k \cdot_G (U_4\cap X)^*}.\]% = \displaystyle{\bigcup_{k<p} (U_3\cap X')^* \cdot_G \alpha_k}$.

Note that for any $x^*\in G$, there is $k<p$ such that $x^*\in
\alpha_k \cdot_G (U_4\cap X)^*$ and then $x^*\cdot_G (U_4\cap
X)^{*}\subseteq \alpha_k \cdot_G (U_3 \cap X)^{*}$.

Let $m$ be the smallest integer such that every $\alpha_i$ is the
$\cdot_G$-product of $m$ elements in $(U_3\cap X)^*$.
\bigskip

\begin{claim} For each $i$, the conjugation map $f_{i}: x\mapsto
\pi_2(\pi_1^{-1}(\alpha_i\cdot_G x^*\cdot_G (\alpha_i)^{-1}))$ is
an algebraic map from $U_{k+m}\cap X$ to $U_{k}\cap X$ for
$k\geq 3$.
\end{claim}

\begin{claimproof} Let $\alpha_i=d_1^*\cdot_G \dots \cdot_G d_m^*$, with each $d_l$ in $X$. Then for any $l$,
for any $j\geq 3$ and $x\in U_j\cap X$ the map $x\mapsto d_l^{-1} x d_l$
is algebraic (as $H$ is algebraic) and by
hypothesis $d_l^{-1} x d_l\in U_{j-1}\cap X$. Since
\[
\pi_2(\pi_1^{-1}(\alpha_i x^* (\alpha_i)^{-1}))=d_l\cdot \dots
\cdot d_1\cdot x \cdot d_1^{-1} \cdot \dots \cdot d_l^{-1},
\]
the function $f_i$ is algebraic as a composition of algebraic functions.
\end{claimproof}

Select points $\{b_i :i<l\}$ in $U_3\cap X$ such that \[U_3\cap X
=\displaystyle{\bigcup_{i<l} b_i \cdot (U_{m+3}\cap X)}.\]

For $j<p$ and $r<l$, define $\alpha_{(j,r)}\in \{\alpha_k:k<p\}$ and $t_{(j,r)}\in (U_3 \cap X)$ such that (where all the products are in $G$):
\[
\alpha_j^{-1} b_r^{*}  \alpha_j = \alpha_{(j,r)} t^{*}_{(j,r)}.
\]

Let $W=(U_3 \cap X) \times \{0,\ldots, p-1\}$ and for $k<p$,
define $W_k = (U_3\cap X) \times \{k\}$.

Define an equivalence relation $E$ on $W^2$ by $(x,i)E(y,j)$ if $\alpha_i\cG x^* = \alpha_j \cG y^*$. We then have
$$(x,i)E(y,j) \iff (y^*)\cG (x^*)^{-1} = \alpha_j^{-1}\cG \alpha_i.$$
If this happens, then $\alpha_j^{-1}\cG \alpha_i$ lies in
$(U_2\cap X)^*$ and can be written as $w_{ij}^*$ for some
$w_{ij}\in U_2 \cap X$. When this is not the case, say that
$w_{i,j}$ is undefined.

Note that we have a definable bijection $\phi: W/E\to G$ sending
$(x,i)$ to $\alpha_i \cG x^*$.

\bigskip

We will now define
a multi-semialgebraic group, which in a way will be the
$\tau$-topological closure of $W/E$. The topology $\tau'$ of $G$ will be induced by the above bijection.

Let $W^{cl} = U_3 \times \{0,\ldots,p-1\}$ and $W^{cl}_k =
U_3 \times \{k\}$. We equip each $W^{cl}_k$ with the
$\tau$-topology. Then $W_k$ is dense in $W^{cl}_k$.

We now define a relation $E^{cl}$ on $W^{cl}$ as follows: given
$(x,i),(y,j)\in W^{cl}$, we have $(x,i)E^{cl}(y,j)$ if and only if
$w_{ij}$ is defined and $yx^{-1} = w_{ij}$.

\begin{claim} $E^{cl}$ is an equivalence relation.
\end{claim}
\begin{claimproof} Reflexivity holds as $w_{ii}=e$ for all $i$. Whenever $w_{ij}$ is defined, then so is $w_{ji}$ and $w_{ji}=w_{ij}^{-1}$. This implies symmetry. Finally, assume that $(x,i)E^{cl} (y,j)$ and $(y,j)E^{cl} (z,k)$, then $zx^{-1} = w_{jk}w_{ij} \in U_2 \cap X$ (as $zx^{-1}\in U_2$ and $w_{jk}w_{ij}\in X$). Then $w_{ik}$ is defined and equal to $w_{jk}w_{ij}$ and thus $(x,i)E^{cl} (z,k)$.
\end{claimproof}

By construction $W/E$ embeds in $W^{cl}/E^{cl}$. We now define a
group structure on $W^{cl}/E^{cl}$. First consider
$(x,i),(y,j),(z,k)\in W$ and write $x=b_r w$ with $w\in
U_{m+3}\cap X$. We then have, where all the products are
understood in $G$:
\begin{align*}
& & \alpha_i x^*  \alpha_j y^*  &= \alpha_k  z^* \\
&\iff &  \alpha_i  b_r^*  w^* \alpha_j  y^* &= \alpha_k  z^* \\
&\iff & \alpha_i  \alpha_j  \alpha_{(j,r)} t^*_{(j,r)} f_j(w)^* y^* &= \alpha_k z^* \\
&\iff & t_{(j,r)}^* f_j(w)^*y^*(z^*)^{-1} &= \alpha_{(j,r)}^{-1} \alpha_j^{-1} \alpha_i^{-1} \alpha_k.
\end{align*}

When such an equation holds, we define $\epsilon(i,j,k,r)$ as $\alpha_{(j,r)}^{-1} \alpha_j^{-1}\alpha_i^{-1}\alpha_k\in U_1\cap X$. Let $\Gamma\in W^3$ be the pullback of the graph of multiplication on $W/E \cong G$ via the canonical projection. Then $((x,i),(y,j),(z,k))\in \Gamma$ if and only if $\epsilon(i,j,k,r)$ is defined and writing $x=b_rw$, we have:
\begin{equation}
\tag{$E\Gamma$}
t_{(j,r)} f_j(w)yz^{-1} = \epsilon(i,j,k,r).
\end{equation}

We define $\Gamma^{cl}$ on $W^{cl}$ by
$((x=b_rw,i),(y,j),(z,k))\in \Gamma^{cl}$ if $(E\Gamma)$ holds. We
need to check that this is well defined, {\it i.e.}, does not
depend on the decomposition of $x$ as $b_rw$. So assume that
$x=b_rw=b_s w'$. Then $w' = b_s^{-1} b_r w$. Assume that
$t_{(j,r)} f_j(w)yz^{-1} = \epsilon(i,j,k,r)$. On a small
neighborhood of $(w,y,z)$ we can find $(w_0,y_0,z_0)$, all points
lying in $X$ such that $t_{(j,r)}
f_j(w_0)y_0z_0^{-1}=\epsilon(i,j,k,r)$ (as all operations are
continuous). Set $w'_0 =  b_s^{-1} b_r w_0$, then $w'_0$ is close
to $w'$, hence in $U_{m+3}\cap X$ and we have $t_{(j,s)}
f_j(w'_0)y_0z_0^{-1}=\epsilon(i,j,k,s)$ (in particular
$\epsilon(i,j,k,s)$ is defined). Letting $(w_0,y_0,z_0)$ converge
to $(w,y,z)$, we obtain $t_{(j,s)}
f_j(w')yz^{-1}=\epsilon(i,j,k,s)$ as required.

A similar argument shows that $\Gamma^{cl}$ is
$E^{cl}$-equivariant: if say $(z,k)E^{cl} (z',k')$, then we have
$z'= w_{ii'}z$ and we conclude as above that $((x,i),(y,j),(z,k))$
is in $\Gamma^{cl}$ if and only if $((x,i),(y,j),(z',k'))$ is in
$\Gamma^{cl}$. Therefore $\Gamma^{cl}$ induces a ternary relation
on the quotient $W^{cl} / E^{cl}$. Note that on each
$W^{cl}_i\times W^{cl}_j \times W^{cl}_k$, $\Gamma^{cl}$ is the
closure of $\Gamma$.

\begin{claim} $\Gamma^{cl}$ induces the graph of a function $W^{cl}/E^{cl} \times W^{cl}/E^{cl}  \to W^{cl}/E^{cl}$.
\end{claim}

\begin{claimproof} First, assume that $(x,i),(y,j) \in W^{cl}$, $x=b_rw$. Then for a given $j$, the equation
$t_{(j,r)} f_j(w)yz^{-1}=\epsilon(i,j,k,r)$ can have at most one
solution in $z$. If we can find $(z',k')$ such that $t_{(j,r)}
f_j(w)yz'^{-1}=\epsilon(i,j,k',r)$ also holds, then
$\epsilon(i,j,k',r)w_{kk'}=\epsilon(i,j,k,r)$ and so
$z'^{-1}w_{kk'}=z^{-1}$ which implies $(z,k)E^{cl} (z',k')$. This
shows that the image is unique.

It remains to show existence. Take $(x,i),(y,j)\in W^{cl}$. Take a
small neighborhood $U_*$ of the identity included in
$\overline{K_H}$. Then there are some $k$  and $r$ such that for any
$x_0\in xU_*$ and $y_0\in yU_*$, there is $(z_0,k)$ with
$((x_0,i),(y_0,j),(z_0,k))\in \Gamma$ and $x_0$ can be written as
$b_r w_0$ with $w_0\in U_{m+4}\cap X$. We may also assume that for
any such $z_0$, $z_0\overline{K_H} \subseteq U_3$. Then we have
$t_{(j,r)} f_j(w_0)y_0z_0^{-1}=\epsilon(i,j,k,r)$. We can then
write $x=b_r w$ for some $w\in U_{m+3}\cap X$ and define
$z=\epsilon(i,j,k,r)^{-1}t_{(j,r)} f_j(w)y$. Then $z\in
z_0\overline{K_H} \subseteq U_3$ and $((x,i),(y,j),(z,k))\in
\Gamma$.
\end{claimproof}

Let $\odot$ the boolean function induced by $\Gamma$ on
$W^{cl}/E^{cl}$. As associativity is a closed condition
$\Gamma^{cl}$ is the closure on $\Gamma$ on each $W_i\times W_j
\times W_k$, $\odot$ is associative. Existence of inverses is
proved as the existence part of the previous claim, fixing $z=e$
and looking for $y$. Therefore we have equipped $W^{cl}/E^{cl}$
with a group structure. This group is $H'$.

The sets $W^{cl}_k$ are multi semialgebraic and $E^{cl}$ and
$\Gamma^{cl}$ are algebraic. One can easily show that $H'$ is in
fact multi semialgebraic (or is in bijection with a multi
semialgebraic group) with underlying set
\[W_0 \cup (W_1\setminus \pi^{-1}(\pi(W_0))) \cup (W_2 \setminus
\pi^{-1}(\pi(W_1)\cap \pi(W_0))) \cup \cdots.\]

Then $G$ embeds definably into $H'$ as the subgroup $W/E$ of finite
index, which completes the proof of the theorem.
\end{proof}

\subsection{Additional comments}

\begin{itemize}
\item Will Johnson has studied in \cite{joh_prc} the model companion of fields with $n$ distinct orderings. This is a particular case of bounded PRC fields. Johnson proves that a Lascar-invariant quantifier-free type extends to a Lascar-invariant measure. It seems likely that an adaptation of those results should show that in this case, any group with f-generics has a translation-invariant measure.
\item We expect those results to generalize to groups definable in the main sort of a pseudo p-adically closed field. This will be dealt with in future work.
\end{itemize}

\appendix

\section{Shelah expansion and NTP$_2$}

\begin{thm}\label{th_shelahexp}
Let $T$ be NTP$_2$ in a language $L$ and assume that we have an expansion $T'$ of $T$ to a  language $L'$ by externally definable sets. Assume furthermore that $T'$ has elimination of quantifiers in $L'$ and the only additional predicates in $L'$ are traces of externally definable NIP formulas. Then $T'$ is NTP$_2$.
\end{thm}
\begin{proof}
Let $M\models T'$ be $\aleph_1$-saturated and let $M\prec N$ be $|M|^+$-saturated. The property of NTP$_2$ for formulas is preserved by finite disjunctions, but not by finite conjunctions in general. It is enough to show that a formula of the form $\phi(x;y)\wedge \psi(x;y)$ is NTP$_2$, where $\phi(x;y)\in L$ and $\psi(x;y) \in L'$ is such that there is an NIP $L$-formula $\tilde \psi(x;y;d)\in L(N)$ such that $\psi(M)=\tilde \psi(N;d)\cap M$. Let $(N,M)$ denote the expansion of $N$ with a new unary predicate naming $M$. Let $(N_1,M_1)\succ (N,M)$ be a sufficiently saturated elementary extension. By Proposition 1.7 in \cite{ChSi2}, there is $\theta(x,y;e)\in L(M_1)$ such that $\theta(M;e)=\psi(M)$ and $\theta(M_1;e)\subseteq \tilde \psi(M_1;d)$. Note that the formula $\theta(x,y;e)$ could have IP.

Now assume that we are given a witness of TP$_2$ for $\phi(x;y)\wedge \psi(x;y)$. Namely, we have an array $(b_{i,j}:i,j<\omega)$ and some $k$ such that each line $\{\phi(x;b_{i,j_0})\wedge \psi(x;b_{i,j_0}):i<\omega\}$ is $k$-inconsistent and for every $\eta:\omega \to \omega$, the path $\{\phi(x;b_{\eta(j),j})\wedge \psi(x;b_{\eta(j),j}):j<\omega\}$ is consistent, hence realized by some $a_\eta\in M$.
Note that by elementarity of the extension $(N_1,M_1)\succ (N,M)$, for each $i$, the intersection of any $k$ formulas in
 $\{\phi(x;b_{i,j_0})\wedge \tilde \psi(x,b_{i,j_0};d):i<\omega\}$ with $M_1$ is empty.
Now the properties of the array are preserved if we replace the formula $\phi(x;y)\wedge \psi(x;y)$ by $\phi(x,y)\wedge \theta(x,y;e)$: the paths are still consistent, using the same witnesses $a_\eta$, and the lines are $k$-inconsistent (in the structure $M_1$) since $\theta(M_1;e)\subseteq \tilde \psi(M_1;d)$. This shows that the formula $\phi(x,y)\wedge \theta(x,y;e)$ has TP$_2$ in $M_1$ which contradicts the hypothesis that $T$ is NTP$_2$.
\end{proof}

\bibliography{PRC}

\end{document}